\documentclass[11pt]{article}

\usepackage[T1]{fontenc}
\usepackage{lmodern}
\usepackage[utf8]{inputenc}
\usepackage{microtype}
\usepackage{framed}
\usepackage{listings}
\usepackage{vmargin}
\usepackage{setspace}
\usepackage{mathrsfs, mathenv}
\usepackage{amsmath, amsthm, amssymb, amsfonts, amscd}
\usepackage{graphicx}
\usepackage{epstopdf}
\usepackage[svgnames]{xcolor}
\usepackage{hyperref}
\hypersetup{citecolor=blue, colorlinks=true, linkcolor=black}
\usepackage[ruled, vlined]{algorithm2e}
\usepackage[capitalise]{cleveref}
\usepackage{subcaption}
\usepackage{bbm}
\usepackage{cite}
\usepackage{verbatim}
\usepackage{pgfplots}
\usepackage{tikz}
\usetikzlibrary{arrows,decorations.pathmorphing,backgrounds,positioning,fit,matrix}
\usepackage{etoolbox}
\usepackage{color}
\usepackage{lipsum}
\usepackage{ifthen}
\usepackage[title]{appendix}
\usepackage{autonum}
\usepackage{accents}
\usepackage{xpatch}
\usepackage[]{algpseudocode}
\usepackage{multicol}
\usepackage[abs]{overpic}
\usepackage{eqnarray}

\crefname{assumption}{Assumption}{Assumptions}
\crefname{figure}{Figure}{Figures}

\theoremstyle{plain}
\newtheorem{theorem}{Theorem}[section]
\newtheorem{corollary}[theorem]{Corollary}
\newtheorem{lemma}[theorem]{Lemma}
\newtheorem{proposition}[theorem]{Proposition}
\numberwithin{equation}{section}

\theoremstyle{definition}

\theoremstyle{remark}
\newtheorem{remark}[theorem]{Remark}
\newtheorem{assumption}[theorem]{Assumption}

\ifpdf
  \DeclareGraphicsExtensions{.eps,.pdf,.png,.jpg}
\else
  \DeclareGraphicsExtensions{.eps}
\fi

\usepackage{mathtools}

\usepackage{booktabs}

\usepackage{pgfplots}
\usepackage{tikz}
\usetikzlibrary{patterns,arrows,decorations.pathmorphing,backgrounds,positioning,fit,matrix}
\usepackage[labelfont=bf]{caption}
\usepackage{here}
\usepackage[font=normal]{subcaption}

\usepackage{enumitem}
\setlist[itemize]{leftmargin=.5in}
\setlist[enumerate]{leftmargin=.5in,topsep=3pt,itemsep=3pt,label=(\roman*)}

\newcommand{\email}[1]{\href{#1}{#1}}
\newcommand{\TheTitle}{Linearization of ergodic McKean SDEs and applications}
\newcommand{\TheAuthors}{G. A. Pavliotis, A. Zanoni}
\title{\TheTitle}
\author{Grigorios A. Pavliotis \thanks{Department of Mathematics, Imperial College London, London SW7 2AZ, UK, \email{g.pavliotis@imperial.ac.uk}.}
\and Andrea Zanoni \thanks{Centro di Ricerca Matematica Ennio De Giorgi, Scuola Normale Superiore, Pisa, Italy, \email{andrea.zanoni@sns.it}.}}
\date{}

\usepackage{amsopn}

\newcommand{\abs}[1]{\left\lvert#1\right\rvert}
\newcommand{\norm}[1]{\left\|#1\right\|}

\newcommand{\Z}{\mathbb{Z}}

\newcommand{\R}{\mathbb{R}}
\newcommand{\T}{\mathbb{T}}

\newcommand{\epl}{\varepsilon}

\newcommand{\defeq}{\coloneqq}
\newcommand{\eqdef}{\eqqcolon}

\newcommand{\E}{\operatorname{\mathbb{E}}}

\newcommand{\argmax}{\operatorname{argmax}}

\renewcommand{\d}{\mathrm{d}}
\newcommand{\dd}{\,\mathrm{d}}
\definecolor{shade}{RGB}{100, 100, 100}
\definecolor{bordeaux}{RGB}{128, 0, 50}

\usepackage[usestackEOL]{stackengine}

\definecolor{leg1}{RGB}{0,114,189}
\definecolor{leg2}{RGB}{217,83,25}
\definecolor{leg3}{RGB}{237,177,32}
\definecolor{leg4}{RGB}{126,47,142}
\definecolor{leg5}{RGB}{119,172,48}

\definecolor{leg21}{RGB}{62,38,169}
\definecolor{leg22}{RGB}{46,135,247}
\definecolor{leg23}{RGB}{55,200,151}
\definecolor{leg24}{RGB}{254,195,56}

\ifpdf
\hypersetup{
	pdftitle={\TheTitle},
	pdfauthor={\TheAuthors}
}
\fi

\begin{document}
	
\maketitle	

\begin{abstract} 
In this article, we consider McKean stochastic differential equations, as well as their corresponding McKean--Vlasov partial differential equations, which admit a unique stationary state, and we study the \emph{linearized} Itô diffusion process that is obtained by replacing the law of the process in the convolution term with the unique invariant measure. We show that the law of the nonlinear McKean process converges to the law of this linearized process exponentially fast in time, both in relative entropy and in Wasserstein distance. We study the problem in both the whole space and the torus. We then show how we can employ the resulting linear (in the sense of McKean) Markov process to analyze properties of the original nonlinear and nonlocal dynamics that depend on their long-time behavior. In particular, we propose a linearized maximum likelihood estimator for the nonlinear process which is asymptotically unbiased, and we study the joint diffusive-mean field limit of the underlying interacting particle system.
\end{abstract}

\textbf{AMS subject classifications.} 26D10, 35Q70, 35Q83, 60J60, 82C22.

\textbf{Key words.} Diffusive limit, interacting particle systems, invariant measure, logarithmic Sobolev inequality, maximum likelihood estimation, McKean--Vlasov PDE.

\section{Introduction}

Interacting particle systems and their mean field limit have been an active area of research in the last few decades, due to their many applications to problems in physics \cite{BiT08,Gol16}, biology \cite{Suz05}, social sciences \cite{Daw83,GPY17,GGS21}, algorithms for sampling and optimization \cite{CHS22,CJZ24} and, more recently, machine learning~\cite{SiS20,MMN18,RoV22, BPA24}. Starting from earlier fundamental works, e.g.,~\cite{Oel84,Szn91}, in recent years there have been important improvements on the understanding of uniform propagation of chaos and its connection to the logarithmic Sobolev inequality (LSI) \cite{Mal01,CDP20,LaL23,GLM24,MRW24,DeT24,DGP23}, nonuniqueness of stationary states \cite{MoR24}, and phase transitions \cite{CGP20,DGP21}, as well as the study of fluctuations around the mean field limit \cite{KLO12} and the Dean--Kawasaki equation \cite{Dea96,KaW98,KLR19,GGK22}. Moreover, there have been several studies on the inference problem for the McKean stochastic differential equation (SDE) and the McKean--Vlasov partial differential equation (PDE). Different approaches, based on maximum likelihood \cite{DeH23}, stochastic gradient descent in continuous time \cite{SKP23}, spectral theoretic (e.g., martingale eigenfunction estimator) techniques~\cite{PaZ22}, deconvolution techniques \cite{ABP24}, the generalized method of moments \cite{PaZ24,CGL24}, have been introduced and analyzed. In addition, fully Bayesian nonparametric inverse problems for the McKean--Vlasov PDE have been considered in \cite{NPR24}. An interesting aspect of the inference problem for the McKean SDE, when the measurement model consists of observations of the particle system, is the nonuniqueness of stationary states for the mean field PDE at high interaction strengths (equivalently, low temperatures) if the confining and interaction potentials are not convex on the whole space \cite{Zha23} or not $H$-stable on the torus \cite{CGP20,DGP21}. This makes the use of inference methodologies that rely on the ergodic properties of the dynamics problematic. This issue was addressed in~\cite{PaZ22} for the Desai--Zwanzig model by first inferring the correct invariant measure and then linearizing the generator of the McKean dynamics around it, and in~\cite{NPR24} by considering observations only at short times, for appropriately chosen initial conditions. One of the main goals of this paper is to further develop the linearization methodology introduced in~\cite{PaZ22}, in the context of parametric inference for McKean SDEs.

In addition, multiscale problems for interacting particle systems and McKean SDEs have been studied in recent years~\cite{GoP17,DGP21,BeS23}. In particular, in~\cite{DGP21} the joint diffusive-mean field limit for systems of weakly interacting reversible diffusions with periodic confining and interaction potentials was studied. It was shown that the diffusive and mean field limits commute when the mean field system restricted on the torus has a unique invariant measure. On the other hand, it was rigorously established that the two limits do not commute when the dynamics on the torus exhibits phase transitions. Furthermore, the covariance matrix of the limiting Brownian motion can be calculated using the standard homogenization (Kipnis--Varadhan) formula~\cite[Section 13.6.1]{PaS08}, for a \emph{linear} elliptic cell problem, obtained by calculating the convolution in the generator of the McKean SDE with respect to the Gibbs measure of the dynamics in the torus. In the present work, we show that the diffusive limit of the interacting particle system and the formula for the diffusion coefficient can be obtained from the linearization procedure.

In this paper we consider the dynamics of a system of $N$ weakly interacting reversible Langevin diffusions, described by the following system of Itô SDEs
\begin{equation} \label{eq:interacting_particle_system}
\begin{aligned}
\d X_t^{(n)} &= - \nabla V(X_t^{(n)}) \dd t - \frac1N \sum_{i=1}^{N} \nabla W(X_t^{(n)} - X_t^{(i)}) \dd t + \sqrt{2 \beta^{-1}} \dd B_t^{(n)}, \\
X_0^{(n)} &\sim \mu_0, \qquad n = 1, \dots, N,
\end{aligned}
\end{equation}
where $V,W \in \mathcal C^2(\mathbb D)$ are the confining and interacting potentials, respectively, $\beta > 0$ is the inverse temperature, and $\mu_0 \in \mathcal P(\mathbb D)$ is the initial distribution which is common among all the particles. The domain $\mathbb D$ can be the entire space $\R^d$ or the multidimensional torus $\T^d$. Under suitable conditions on the confining and interaction potentials, a rather complete description of the propagation of chaos phenomenon has been achieved, as well as of the properties of the mean field PDE~\cite{Mal01,LaL23,DGP23}. In particular, in the limit as $N \to \infty$, particles become decorrelated and $X_t^{(n)}$, $n = 1, \dots, N$, converge to the mean field limit
\begin{equation} \label{eq:McKeanSDE}
\d X_t = - \nabla V(X_t) \dd t -  (\nabla W * f_t)(X_t) \dd t + \sqrt{2 \beta^{-1}} \dd B_t, \qquad X_0 \sim \mu_0,
\end{equation} 
where $f_t \colon \mathbb D \to \R$ denotes the density of the law $\mu_t$ of the process $X_t$ with respect to the Lebesgue measure, which satisfies the McKean--Vlasov PDE
\begin{equation} \label{eq:McKeanVlasovPDE}
\frac{\partial f_t}{\partial t} = \nabla \cdot (\nabla V f_t + (\nabla W * f_t) f_t) + \beta^{-1} \Delta f_t.
\end{equation}
The, possibly more than one, invariant measures $\d\mu_\infty(x) = f_\infty(x) \dd x$ of the McKean SDE~\eqref{eq:McKeanSDE} are solutions of the Kirkwood--Monroe integral equation~\cite{Tam84, Dre87}
\begin{equation} \label{eq:integral_equation}
f_\infty = \frac{1}{Z}e^{- \beta( V + (W * f_\infty))}, \qquad Z = \int_{\mathbb D} e^{- \beta( V(z) + (W * f_\infty)(z))} \dd z.
\end{equation}
It is well known that, under the conditions on the interaction and confining potentials that ensure uniform propagation of chaos for the interacting particle system, the McKean SDE is geometrically ergodic, i.e., it converges exponentially fast in $L^1$, relative entropy, and 2--Wasserstein distance to its unique invariant measure (see, e.g., \cite{Mal01} for the dynamics in the whole space $\R^d$). Moreover, a different approach is presented in \cite{EGZ19}, where the authors prove convergence to equilibrium in Wasserstein distance by means of coupling techniques, specifically combination of reflection and synchronous couplings, which enable them to obtain explicit contraction rates. Recent developments in the study of uniform propagation of chaos and of its connection to the uniform LSI can be found in~\cite{LaL23,DGP23}. This idea has also been used in \cite{GLM24} to prove the propagation of chaos on the torus under different assumptions on the interaction potential that replace convexity. Considering the multidimensional torus $\T^d$ as a state space, a different approach based on the (Lions) derivatives of the semigroup generated by the Fokker--Planck equation is proposed in \cite{DeT24}.

Let us now place ourselves in the setting where we have uniform propagation of chaos and, therefore, the mean field dynamics converges exponentially fast to the unique invariant measure. In this case, the McKean SDE \eqref{eq:McKeanSDE} and the \emph{linear} (in the sense of McKean) diffusion process obtained by replacing the density $f_t$ with $f_{\infty}$
\begin{equation} \label{eq:McKeanSDE_linear}
\d Y_t = - \nabla V(Y_t) \dd t -  (\nabla W * f_\infty)(Y_t) \dd t + \sqrt{2 \beta^{-1}} \dd B_t, \qquad Y_0 \sim \nu_0,
\end{equation} 
for some initial condition $\nu_0 \in \mathcal P(\mathbb D)$, are reversible, when started at stationarity, with respect to the same Gibbsian invariant measure $f_\infty$, the unique solution of~\eqref{eq:integral_equation}; see~\cite[Chapter 4]{Pav14}. The exponentially fast convergence of $f_t$ to $f_{\infty}$, suggests that, for arbitrary initial conditions $X_0 \sim \mu_0$ and $Y_0 \sim \nu_0$, we expect that the nonlinear process $X_t$ and the linear process $Y_t$ are exponentially close in time. This suggests that, in order to study the long-term behavior of solutions to the McKean SDE and to the corresponding McKean--Vlasov PDE, it is sufficient to consider the linearized process $Y_t$. This fact was already used implicitly in \cite{PaZ22} when applying the martingale estimating function methodology~\cite{KeS99} to estimate parameters in the drift as well as the diffusion coefficient in equation~\eqref{eq:interacting_particle_system}. In particular, it was shown in~\cite{PaZ22} that, in order to infer parameters in the mean field dynamics, it is sufficient to estimate the eigenpairs of the generator of the linearized dynamics \eqref{eq:McKeanSDE_linear}. This linearization methodology was also used in \cite{DGP21} when studying the joint diffusive-mean field limit of the interacting particle system \eqref{eq:interacting_particle_system} with periodic coefficients. In particular, it was shown that in order to study the diffusive limit of the mean field dynamics, it is sufficient to prove a central limit theorem and, based on this, an invariance principle for the linearized dynamics. Consequently, the formula for the covariance matrix of the limiting Brownian motion requires the solution of the Poisson equation for the linearized process, and the computation of an average with respect to the Gibbs measure of the process on the torus. We reiterate that the nonlinear and linear processes are reversible with respect to the same Gibbs measure.

In this paper we build on the ideas presented in~\cite{DGP21, PaZ22} in order to develop the linearization of the McKean SDE as a systematic tool for studying the mean field limit of weakly interacting diffusions. More specifically, we show that, when confronted with tasks that require the analysis of the long-time behaviour of the dynanics, the nonlinear process $X_t$ can be ``replaced'' by the linear process $Y_t$, whenever the confining and interaction potentials are such that uniform propagation of chaos holds. We obtain quantitative estimates on the distance between the two processes both in relative entropy and Wasserstein distance, considering both the whole space and the torus as state spaces. Moreover, in the former setting, we also show strong convergence in $L^2$ for the paths using a coupling argument. We remark that the problem of linearizing nonlinear Fokker--Planck equations on the whole space has also been studied in \cite{RRW22}, where the authors propose a different method to obtain a linearized process which uses the natural tangent bundle over the space of probabilities with the corresponding gradient operator. Our analysis, instead, is mostly based on entropy estimates between solutions of Fokker--Planck equations presented in \cite{DLP18,LaL23}, based on the important paper~\cite{BRS16}. We then combine these estimates with a logarithmic Sobolev inequality for the Gibbs measure and for the nonlinear and linear semigroups. In particular, we have that, for $X_t \sim \mu_t$ and $Y_t \sim \nu_t$, 
\begin{equation}
\mathcal H (\mu_t | \nu_t) + \frac1{2\beta} \int_0^t  \mathcal I(\mu_s | \nu_s) \dd s  \le \mathcal H (\mu_0 | \nu_0)  + \frac\beta2 \int_0^t \int_{\mathbb D} \norm{(\nabla W * (\mu_s - \mu_\infty))(x)}^2 \mu_s(x) \dd x \dd s,
\end{equation}
where $\mathcal H$, $\mathcal I$ denote the relative entropy and relative Fisher information, respectively. This, together with the LSI, implies
\begin{equation}
\mathcal H (\mu_t | \nu_t) + \frac1{2\beta\lambda} \int_0^t  \mathcal H(\mu_s | \nu_s) \dd s  \le \mathcal H (\mu_0 | \nu_0)  + \frac\beta2 \int_0^t \int_{\mathbb D} \norm{(\nabla W * (\mu_s - \mu_\infty))(x)}^2 \mu_s(x) \dd x \dd s,
\end{equation}
for some appropriate constant $\lambda > 0$. The details can be found in~\cref{sec:analysis}. The fact that the laws $\mu_t$ and $\nu_t$ are exponentially close follows from this estimate.

We consider the problem both in $\R^d$ and in the torus $\T^d$. In the latter case, we extend some of the results obtained in~\cite{CGP20, GLM24} on convergence to equilibrium and the LSI to the case where a confining potential in present, in addition to the interaction potential. This is one of the novelties of the present work. Then, we apply our methodology to propose a maximum likelihood estimator (MLE) for the mean field SDE \eqref{eq:McKeanSDE} which is based on the likelihood of the linearized process $Y_t$, and we prove its asymptotic unbiasedness in the limit of infinite observations. Finally, we also revisit the combined diffusive-mean field limit for interacting diffusions in periodic confining and interaction potentials that was studied in \cite{DGP21}, and we calculate the variance of the limiting Brownian motion using the linearized process.

The main contributions of this work can be summarized as follows.
\begin{itemize}[leftmargin=*]
\item We prove that, under assumptions on the confining and interaction potentials that ensure uniform propagation of chaos, the McKean SDE on $\R^d$ is exponentially close in time to the linearized process, both in relative entropy and in the $2$-Wasserstein distance.
\item We obtain a similar result for the McKean SDE on the torus in relative entropy. As a by-product, we extend some of the known results on the McKean SDE on the torus to the case where an arbitrary confining potential is present.
\item We use the linearized process~\eqref{eq:McKeanSDE_linear} to propose a ``linearized'' MLE estimator for the \emph{nonlinear} McKean process which is asymptotically unbiased. In addition, we study the joint diffusive-mean field limit for weakly interacting diffusions in periodic confining and interaction potentials. Our theoretical results are supplemented by numerical simulations. 
\end{itemize}

The remainder of the paper is organized as follows. In \cref{sec:analysis} we show the convergence analysis for the law of the linearized McKean SDE, which we perform both in the whole space in \cref{sec:analysi_whole} and on the torus in \cref{sec:analysi_torus}. Then, in \cref{sec:applications} we present two applications of the linearized process: maximum likelihood estimation in \cref{sec:MLE} and computation of the asymptotic variance in a central limit theorem in \cref{sec:CLT}. Finally, in \cref{sec:conclusion} we draw our conclusions and outline future developments of the linearization procedure.

\section{The linearized diffusion process} \label{sec:analysis}

\subsection{Preliminaries}

We consider the McKean SDE \eqref{eq:McKeanSDE} on the domain $\mathbb D$ (either $\R^d$ or $\T^d$)  
\begin{equation}
\d X_t = - \nabla V(X_t) \dd t -  (\nabla W * f_t)(X_t) \dd t + \sqrt{2 \beta^{-1}} \dd B_t,
\end{equation}
and its linearization~\eqref{eq:McKeanSDE_linear}
\begin{equation}
\d Y_t = - \nabla V(Y_t) \dd t -  (\nabla W * f_\infty)(Y_t) \dd t + \sqrt{2 \beta^{-1}} \dd B_t.
\end{equation}
We remark that to define the linearized process $Y_t$ we need the following main hypothesis.
\begin{assumption} \label{as:uniqueness}
The integral equation \eqref{eq:integral_equation} has a unique solution $f_\infty$ and therefore the McKean SDE \eqref{eq:McKeanSDE} has a unique invariant measure $\d \mu_\infty(x) = f_\infty(x) \dd x$.
\end{assumption}
Let now $\mu_t$ and $\nu_t$ be the law of the processes $X_t$ and $Y_t$, respectively, and denote their densities with respect to the Lebesgue measure by $f_t$ and $g_t$, which satisfy the Fokker--Planck equations
\begin{equation} \label{eq:FP_equations}
\begin{aligned}
\frac{\partial f_t}{\partial t} &= \nabla \cdot (\nabla V f_t + (\nabla W * f_t) f_t) + \beta^{-1} \Delta f_t, \\
\frac{\partial g_t}{\partial t} &= \nabla \cdot (\nabla V g_t + (\nabla W * f_\infty) g_t) + \beta^{-1} \Delta g_t,
\end{aligned}
\end{equation}
together with the appropriate initial conditions. Note that, by construction, the invariant measures coincide, that is, $\nu_\infty = \mu_\infty$ and $g_\infty = f_\infty$, and the densities are given by the integral equation \eqref{eq:integral_equation}. In the following sections, we quantify the distance between processes $X_t$ and $Y_t$, focusing first on the case where the domain $\mathbb D$ is the whole space $\R^d$ and then on the multidimensional torus $\T^d$. The theoretical results, and in particular the main assumptions that guarantee, e.g., the uniqueness of the invariant measure, are specific for each case and therefore are discussed in each section separately.

\begin{remark}
The situation is more complicated when we do not have uniqueness of the invariant measure. However, \cref{as:uniqueness} is not a strong limitation of the methodology introduced in this work. In fact, in \cite{MoR24} it is shown that when multiple invariant measures can exist, it is still possible to prove exponentially fast convergence to equilibrium locally, i.e., as long as the initial distribution is sufficiently close to the stationary state and we are away from the critical temperature, where we only get algebraic convergence. Therefore, we believe that the results presented here can be extended to the case of multiple invariant measures by selecting the ``correct'' density for the linearization procedure depending on basins of attraction of the different invariant measures. In fact, this idea was already employed in the numerical experiments on inference for the Desai--Zwanzig model in~\cite{PaZ22}. 
\end{remark}

The relative entropy $\mathcal H$ and the relative Fisher information $\mathcal I$ between two distributions $\mu$ and $\nu$ in $\mathcal P(\mathbb D)$ such that $\mu$ is absolutely continuous with respect to $\nu$ are
\begin{equation}
\mathcal H (\mu | \nu) = \int_{\mathbb D} \log \left( \frac{\d \mu}{\d \nu}(x) \right) \dd \mu(x), \qquad
\mathcal I(\mu | \nu) = \int_{\mathbb D} \left| \nabla \log \left( \frac{\d \mu}{\d \nu}(x) \right) \right|^2 \dd \mu(x).
\end{equation}
Moreover, if $\mu$ and $\nu$ admit densities $f$ and $g$, respectively, with respect to the Lebesgue measure, then the relative entropy and Fisher information can be rewritten as
\begin{equation}
\mathcal H (\mu | \nu) = \int_{\mathbb D} \log \left( \frac{f(x)}{g(x)} \right) f(x) \dd x, \qquad
\mathcal I(\mu | \nu) = \int_{\mathbb D} \left| \nabla \log \left( \frac{f(x)}{g(x)} \right) \right|^2 f(x) \dd x.
\end{equation}
We finally recall that, following the notation in \cite{MRW24}, $\nu$ satisfies an LSI with constant $\lambda > 0$ if for all $\mu \in \mathcal P(\mathbb D)$ such that $\mu$ is absolutely continuous with respect to $\nu$ it holds
\begin{equation} \label{eq:LSI_def}
\mathcal H (\mu | \nu) \le \frac\lambda4 \mathcal I(\mu | \nu).
\end{equation}

\subsection{Analysis on $\R^d$} \label{sec:analysi_whole}

We now fix the domain $\mathbb D$ to be the whole space $\R^d$ and we consider the setting of \cite{Mal01}, which guarantees the well-posedness of the McKean SDE and of the corresponding McKean--Vlasov PDE and the uniqueness of the invariant measure due to the convexity of the confining and interaction potentials. In particular, we work under the following assumptions.
\begin{assumption} \label{as:whole}
The initial distributions $\mu_0$ and $\nu_0$ have bounded moments of all orders, and the potentials $V, W \colon \R^d \to \R$ satisfy:
\begin{enumerate}[leftmargin=*]
\item $V, \, W \in \mathcal C^2(\R^d)$ are polynomially bounded along with their derivatives;
\item $W$ is even;
\item\label{it:convex} $V$ and $W$ are uniformly convex with constants $\alpha > 0$ and $\gamma \ge 0$, respectively, which means that for all $x,y \in \R^d$
\begin{equation}
\begin{aligned}
(\nabla V(x) - \nabla V(y))^\top (x - y) &\ge \alpha \abs{x - y}^2, \\
(\nabla W(x) - \nabla W(y))^\top (x - y) &\ge \gamma \abs{x - y}^2.
\end{aligned}
\end{equation}
\end{enumerate}
\end{assumption}
Then, from~\cite[Theorem 1.3]{Mal01} it follows that the density $f_t$ converges to $f_\infty$ in $L^1$ exponentially fast, i.e., there exists a constant $K > 0$ such that for all $t \ge 0$
\begin{equation} \label{eq:convergenceL1_density_whole}
\norm{f_t - f_\infty}_{L^1(\R^d)} \le K e^{-\frac12 \alpha t}.
\end{equation}
Moreover, since $\mu_0$ has finite moments of any order, then so does the law of the process uniformly in time. Therefore, we have the following bound.

\begin{lemma}
Under \cref{as:whole}, there exists $M > 0$ such that for all $t \ge 0$
\begin{equation} \label{eq:M_nablaW_whole}
\int_{\R^d} \int_{\R^d} \abs{\nabla W(x - z)}^2 (f_t(z) + f_\infty(z)) f_t(x) \dd z \dd x \le M.
\end{equation}
\end{lemma}
\begin{proof}
Since $\mu_t$ and $\mu_\infty$ have bounded moments of any order uniformly in time by \cite{Mal01}, the desired result follows from the fact that $\nabla W$ is polynomially bounded.
\end{proof}

In the following lemma we show that convolving the interaction potential with a probability density function preserves the convexity property.

\begin{lemma} \label{lem:convexity_convolution}
Let $W$ satisfy \cref{as:whole}\ref{it:convex}, then, for all probability density functions $f$ of measures in $\mathcal P(\R^d)$, the function $W * f$ is uniformly convex with constant $\gamma \ge 0$, i.e., it holds for all $x,y \in \R^d$
\begin{equation}
((\nabla W * f)(x) - (\nabla W * f)(y))^\top (x - y) \ge \gamma \abs{x - y}^2.
\end{equation}
\end{lemma}
\begin{proof}
By definition of convolution, we have
\begin{equation}
\begin{aligned}
&((\nabla W * f)(x) - (\nabla W * f)(y))^\top (x - y) \\
&\qquad = \left( \int_{\R^d} (\nabla W(x-z) - \nabla W(y-z)) f(z) \dd z \right)^\top (x-y) \\
&\qquad = \int_{\R^d} (\nabla W(x-z) - \nabla W(y-z))^\top ((x-z) - (y-z)) f(z) \dd z,
\end{aligned}
\end{equation}
which due to \cref{as:whole}\ref{it:convex} implies
\begin{equation}
((\nabla W * f)(x) - (\nabla W * f)(y))^\top (x - y) \ge \int_{\R^d} \gamma \abs{x-y}^2 f(z) \dd z = \gamma \abs{x - y}^2,
\end{equation}
which is the desired result.
\end{proof}

\subsubsection{Convergence of the linearized process}

In the classical work \cite{Mal01} as well in the interesting recent paper \cite{MRW24}, it is shown that, provided that the initial condition satisfies an LSI, and under our convexity assumptions on the confining and interaction potentials given in \cref{as:whole}, then so does the solution of the McKean--Vlasov PDE \eqref{eq:McKeanVlasovPDE}, with an LSI constant that can be explicitly computed. In the following lemma we show that the same conclusion holds for the linearized McKean--Vlasov PDE, with the same LSI constant.

\begin{lemma} \label{lem:LSI_whole}
Let \cref{as:whole} be satisfied and assume that $\mu_0$ and $\nu_0$ satisfy an LSI with constant $\lambda_0 > 0$. Then $\mu_t$ and $\nu_t$ satisfy an LSI with constant 
\begin{equation}
\lambda_t = \lambda_0 e^{-2(\alpha + \gamma)t} + \frac1{2\beta(\alpha + \gamma)} \left( 1 - e^{-2(\alpha + \gamma)t} \right).
\end{equation}
Moreover, $\mu_t$ and $\nu_t$ satisfy a time-uniform LSI with constant 
\begin{equation}
\Lambda = \max \left\{ \lambda_0, \frac1{2\beta(\alpha + \gamma)} \right\}.
\end{equation}
\end{lemma}
\begin{proof}
The desired result follows directly from \cite[Proposition 1.1]{MRW24} and \cref{lem:convexity_convolution}, and the uniform in time constant is obtained by noticing that $\lambda_t$ is a convex combination of $\lambda_0$ and $1/(2\beta(\alpha + \gamma))$.
\end{proof}

From the LSI we can now deduce the exponentially fast convergence in relative entropy of the law of the linearized process $\nu_t$ to the law of the nonlinear process $\mu_t$. 

\begin{theorem} \label{thm:entropy_whole}
Let the assumptions of \cref{lem:LSI_whole} be satisfied. Then, it holds
\begin{equation}
\mathcal H (\mu_t | \nu_t) \le \rho_\Lambda(t) \defeq \begin{cases}
\left( \mathcal H (\mu_0 | \nu_0) + \frac{\beta^2 MK \Lambda}{4 - \alpha \beta \Lambda} \right) e^{-\frac12 \alpha t}, & \text{if } \Lambda < \frac4{\alpha\beta}, \\
\left( \mathcal H (\mu_0 | \nu_0) + \frac12 \beta MK t \right) e^{- \frac2{\beta\Lambda} t}, & \text{if } \Lambda = \frac4{\alpha\beta}, \\
\left( \mathcal H (\mu_0 | \nu_0) + \frac{\beta^2 MK \Lambda}{\alpha \beta \Lambda - 4} \right) e^{- \frac2{\beta\Lambda} t}, & \text{if } \Lambda > \frac4{\alpha\beta}.
\end{cases}
\end{equation}
\end{theorem}
\begin{proof}
By \cite[Lemma 3.1]{LaL23} we have the estimate
\begin{equation}
\mathcal H (\mu_t | \nu_t) + \frac1{2\beta} \int_0^t  \mathcal I(\mu_s | \nu_s) \dd s \le \mathcal H (\mu_0 | \nu_0) + \frac\beta2 \int_0^t \int_{\R^d} \abs{(\nabla W * (f_s - f_\infty))(x)}^2 f_s(x) \dd x \dd s,
\end{equation}
which due to \cref{lem:LSI_whole} implies
\begin{equation} \label{eq:entropy_initial_estimate}
\mathcal H (\mu_t | \nu_t) + \frac2{\beta\Lambda} \int_0^t  \mathcal H(\mu_s | \nu_s) \dd s \le \mathcal H (\mu_0 | \nu_0) + \frac\beta2 \int_0^t \int_{\R^d} \abs{(\nabla W * (f_s - f_\infty))(x)}^2 f_s(x) \dd x \dd s.
\end{equation}
By Cauchy--Schwarz inequality we have
\begin{equation}
\abs{(\nabla W * (f_s - f_\infty))(x)}^2 \le \left( \int_{\R^d} \abs{\nabla W(x - z)}^2 (f_s(z) + f_\infty(z)) \dd z \right) \norm{f_s - f_\infty}_{L^1(\R^d)},
\end{equation}
and due to the bound \eqref{eq:M_nablaW_whole} and the convergence to the invariant measure given in \eqref{eq:convergenceL1_density_whole} we obtain
\begin{equation}
\int_{\R^d} \abs{(\nabla W * (f_s - f_\infty))(x)}^2 f_s(x) \dd x \le M \norm{f_s - f_\infty}_{L^1(\R^d)} \le M K e^{-\frac12 \alpha s}.
\end{equation}
Therefore, estimate \eqref{eq:entropy_initial_estimate} becomes
\begin{equation}
\mathcal H (\mu_t | \nu_t) \le \mathcal H (\mu_0 | \nu_0) + \frac{\beta MK}2 \int_0^t e^{-\frac12 \alpha s} \dd s - \frac2{\beta\Lambda} \int_0^t  \mathcal H(\mu_s | \nu_s) \dd s,
\end{equation}
which, by the Grönwall's inequality in \cite[Lemma A.1]{LaL23}, implies 
\begin{equation}
\begin{aligned}
\mathcal H (\mu_t | \nu_t) &\le \mathcal H (\mu_0 | \nu_0) e^{- \frac2{\beta\Lambda} t} + \frac{\beta MK}2 \int_0^t e^{- \frac2{\beta\Lambda} (t-s)} e^{-\frac12 \alpha s} \dd s, \\
&= \mathcal H (\mu_0 | \nu_0) e^{- \frac2{\beta\Lambda} t} + \frac{\beta MK}2  e^{- \frac2{\beta\Lambda} t} \int_0^t e^{\frac{4 - \alpha \beta \Lambda}{2 \beta \Lambda} s} \dd s
\end{aligned}
\end{equation}
Then, if $\Lambda = 4/(\alpha\beta)$ we deduce
\begin{equation}
\mathcal H (\mu_t | \nu_t) \le \left( \mathcal H (\mu_0 | \nu_0) + \frac{\beta MK}2 t \right) e^{- \frac2{\beta\Lambda} t}.
\end{equation}
On the other hand, if $\Lambda \neq 4/(\alpha\beta)$ we have
\begin{equation}
\mathcal H (\mu_t | \nu_t) \le \mathcal H (\mu_0 | \nu_0) e^{- \frac2{\beta\Lambda} t} + \frac{\beta^2 MK \Lambda}{4 - \alpha \beta \Lambda} \left( e^{-\frac12 \alpha t} - e^{- \frac2{\beta\Lambda} t} \right),
\end{equation}
and the desired result follows.
\end{proof}

The following convergence in Wasserstein distance is a direct consequence of the Talagrand’s transportation inequality.

\begin{corollary} \label{cor:wasserstein_whole}
Let the assumptions of \cref{thm:entropy_whole} be satisfied. Then, it holds
\begin{equation}
\mathcal W_2 (\mu_t, \nu_t) \le \sqrt{ \Lambda \rho_\Lambda(t)},
\end{equation}
where $\mathcal W_2$ denotes the 2-Wasserstein distance and $\rho_\Lambda(t)$ is defined in \cref{thm:entropy_whole}.
\end{corollary}
\begin{proof}
The desired result follows from \cref{thm:entropy_whole} using \cite[Theorem 1]{OtV00}.
\end{proof}

Finally, we also have strong convergence of the paths $X_t$ and $Y_t$ in $L^2$.

\begin{theorem} \label{thm:coupling_whole}
Let \cref{as:whole} be satisfied. Then, it holds
\begin{equation}
\E[\norm{X_t - Y_t}^2] \le \left( \E[\norm{X_0 - Y_0}^2] + \frac{8MK}{(3\alpha + 4\gamma)^2} \right) e^{-\frac12 \alpha t}.
\end{equation}
\end{theorem}
\begin{proof}
Combining equations \eqref{eq:McKeanSDE} and \eqref{eq:McKeanSDE_linear}, we have
\begin{equation}
\frac{\d}{\d t} (X_t - Y_t) = - (\nabla V(X_t) - \nabla V(Y_t)) - ((\nabla W * f_t)(X_t) - (\nabla W * f_\infty)(Y_t)),
\end{equation}
which, multiplying by $(X_t - Y_t)$, implies
\begin{equation}
\begin{aligned}
\frac12 \frac{\d}{\d t} \abs{X_t - Y_t}^2 &= - (\nabla V(X_t) - \nabla V(Y_t))^\top (X_t - Y_t) \\
&\quad - ((\nabla W * f_t)(X_t) - (\nabla W * f_\infty)(Y_t))^\top (X_t - Y_t) \\
&\eqdef I_1 + I_2.
\end{aligned}
\end{equation}
First, by the uniform convexity of $V$ we get
\begin{equation} \label{eq:bound_I1}
I_1 \le - \alpha \abs{X_t - Y_t}^2,
\end{equation}
and then we rewrite $I_2$ as
\begin{equation}
\begin{aligned}
I_2 &= - ((\nabla W * f_t)(X_t) - (\nabla W * f_\infty)(X_t))^\top (X_t - Y_t) \\
&\quad - ((\nabla W * f_\infty)(X_t) - (\nabla W * f_\infty)(Y_t))^\top (X_t - Y_t) \\
&\eqdef I_{2,1} + I_{2,2},
\end{aligned}
\end{equation}
where \cref{lem:convexity_convolution} gives
\begin{equation} \label{eq:bound_I22}
I_{2,2} \le - \gamma \abs{X_t - Y_t}^2.
\end{equation}
Using Young's inequality with $\delta > 0$ we have
\begin{equation}
I_{2,1} \le \frac1{2\delta} \abs{(\nabla W * f_t)(X_t) - (\nabla W * f_\infty)(X_t)}^2 + \frac\delta2 \abs{X_t - Y_t}^2,
\end{equation}
and by Cauchy--Schwarz inequality we obtain
\begin{equation} \label{eq:bound_I21}
I_{2,1} \le \frac\delta2 \abs{X_t - Y_t}^2 + \frac1{2\delta} \left( \int_{\R^d} \abs{\nabla W(X_t - z)}^2 (f_t(z) + f_\infty(z)) \dd z \right) \norm{f_t - f_\infty}_{L^1(\R^d)}.
\end{equation}
Collecting bounds \eqref{eq:bound_I1}, \eqref{eq:bound_I22}, \eqref{eq:bound_I21} and by equations \eqref{eq:convergenceL1_density_whole} and \eqref{eq:M_nablaW_whole} we have
\begin{equation}
\frac{\d}{\d t} \E[\abs{X_t - Y_t}^2] \le - 2 \left( \alpha + \gamma - \frac\delta2 \right) \E[\abs{X_t - Y_t}^2] + \frac{MK}{2\delta} e^{-\frac12 \alpha t},
\end{equation}
which due to Grönwall's inequality implies
\begin{equation}
\begin{aligned}
\E[\abs{X_t - Y_t}^2] &\le  \E[\abs{X_0 - Y_0}^2] e^{- \left( 2\alpha + 2\gamma - \delta \right) t} + \frac{MK}{2\delta} \int_0^t e^{- \left( 2\alpha + 2\gamma - \delta \right) (t-s)} e^{- \frac12 \alpha s} \dd s \\
&= \E[\abs{X_0 - Y_0}^2] e^{- \left( 2\alpha + 2\gamma - \delta \right) t} + \frac{MK}{\delta (3\alpha + 4\gamma - 2\delta)} \left( e^{-\frac12 \alpha t} - e^{- \left( 2\alpha + 2\gamma - \delta \right) t} \right).
\end{aligned}
\end{equation}
Finally, choosing the optimal value $\delta = (3\alpha + 4\gamma)/4$ yields the desired result.
\end{proof}

\begin{remark}
If the LSI constant $\lambda_0$ of the initial distributions $\mu_0$ and $\nu_0$ is sufficiently small, namely $\lambda_0 < 4/(\alpha\beta)$, then the rate of convergence in \cref{thm:entropy_whole,cor:wasserstein_whole} is determined only by the convexity parameter $\alpha$ of the confining potential. This is in agreement with the rate of convergence in $L^2$ predicted by \cref{thm:coupling_whole}.
\end{remark}

\subsection{Analysis on $\T^d$} \label{sec:analysi_torus}

In this section, the domain $\mathbb D$ is the multidimensional torus $\mathbb T^d$ of size length $1$, and we work in the framework below.

\begin{assumption} \label{as:torus}
The potentials $V, W \colon \T^d \to \R$ satisfy:
\begin{enumerate}
\item $V,W \in \mathcal C^2(\T^d)$;
\item $W$ is even;
\item\label{it:stable} W is either $H$-stable, meaning that it has nonnegative Fourier modes \cite{Rue99}, or is such that $\norm{W}_{L^\infty(\T^d)} < \beta^{-1}$.
\end{enumerate}
\end{assumption}

Notice that \cref{as:torus}\ref{it:stable} guarantees that \cref{as:uniqueness} is satisfied. In fact, all invariant probability density functions $f_\infty$ are also minimizers of the free energy
\begin{equation}
\mathcal E(f) = \int_{\T^d} V(x) f(x) \dd x + \frac12 \int_{\T^d} \int_{\T^d} W(x-y) f(y) f(x) \dd y \dd x + \beta^{-1} \int_{\T^d} f(x) \log(f(x)) \dd x,
\end{equation}
which is strictly convex under \cref{as:torus}\ref{it:stable}. This is shown in \cite[Proposition 2.8]{CGP20} without the confining potential, which, however, does not affect the convexity of $\mathcal E$ since its corresponding term is linear; see also~\cite[Section V.B.]{Bav91}. Moreover, the invariant measure $\mu_\infty$ satisfies an LSI, as shown in the following result.

\begin{lemma} \label{lem:LSI_invariant_torus}
Under \cref{as:torus}, the invariant measure $\d \mu_\infty(x) = f_\infty(x) \dd x$ satisfies an LSI with constant $\Gamma / (2\pi^2)$, where
\begin{equation} \label{eq:Gamma_def}
\Gamma = e^{2\beta \left( \norm{V}_{L^\infty(\T^d)} + \norm{W}_{L^\infty(\T^d)} \right)}.
\end{equation}
\end{lemma}
\begin{proof}
We follow~\cite[Lemma 2.4]{GLM24}. From the integral equation \eqref{eq:integral_equation} and the Young's convolution inequality we deduce that for all $x \in \T^d$
\begin{equation}
f_\infty(x) = \frac{e^{- \beta( V(x) + (W * f_\infty)(x))}}{\int_{\T^d} e^{- \beta( V(z) + (W * f_\infty)(z))} \dd z} \le \frac{e^{\beta \left( \norm{V}_{L^\infty(\T^d)} + \norm{W}_{L^\infty(\T^d)} \norm{f_\infty}_{L^1(\T^d)} \right)}}{\int_{\T^d} e^{- \beta \left( \norm{V}_{L^\infty(\T^d)} + \norm{W}_{L^\infty(\T^d)} \norm{f_\infty}_{L^1(\T^d)} \right)} \dd z}.
\end{equation}
Proceeding analogously for the lower bound, and noting that $\norm{f_\infty}_{L^1(\T^d)} = 1$ since $f_\infty$ is a probability density function and that $\T^d$ has Lebesgue measure 1, we obtain
\begin{equation}
\frac1\Gamma = e^{-2\beta \left( \norm{V}_{L^\infty(\T^d)} + \norm{W}_{L^\infty(\T^d)} \right)} \le f_\infty(x) \le e^{2\beta \left( \norm{V}_{L^\infty(\T^d)} + \norm{W}_{L^\infty(\T^d)} \right)} = \Gamma.
\end{equation}
Then, the desired result follows from the Holley–-Stroock perturbation lemma (see, e.g., \cite[Lemma 2.3]{GLM24}) and the fact that the uniform distribution on the torus satisfies an LSI (see, e.g., \cite[Lemma 2.4]{GLM24}). We remark that the discrepancy with \cite[Lemma 2.4]{GLM24} is due to the factor $4$ in the LSI \eqref{eq:LSI_def}.
\end{proof}

\begin{remark}
The results presented in the following sections still hold if, instead of having a unique invariant measure, we have multiple stationary states but a unique stable one. This is the case, e.g, of the model in \cite[Lemma 1.26]{DGP21} above the phase transition. In fact, for all initial conditions except the unstable invariant measure, the law of the McKean SDE converges to the stable stationary state. It is not clear whether it is possible to have a similar situation (precisely two invariant measures, one stable and one unstable) when the state space is the whole space $\R^d$, at least for polynomial interaction and confining potentials~\cite{Tug14}. 
\end{remark}

\subsubsection{Convergence to equilibrium for the McKean--Vlasov PDE}

Before presenting the convergence results for the linearized process, we need the exponentially fast convergence of the measure $\mu_t$ to the invariant measure $\mu_\infty$. This problem has already been studied when there is no confining potential in, e.g., \cite{CJL17,CGP20}. In the absence of confining potential, the analysis is simplified because the unique, under~\cref{as:torus}\ref{it:stable}, invariant measure is the Lebesgue measure. However, here we extend the results to the case of a nonzero confining potential. Given two probabilities $\mu, \nu \in \mathcal P(\T^d)$ with densities $f,g$, we denote their total variation distance by $\mathrm{TV}(\mu,\nu)$. Moreover, we recall the following equivalence with the $L^1(\T^d)$ distance between the densities (see, e.g., \cite[Lemma 2.1]{Tsy09})
\begin{equation} \label{eq:TV_L1}
\mathrm{TV}(\mu,\nu) = \frac12 \norm{f - g}_{L^1(\T^d)}.
\end{equation}
We first consider the $L^2(\T^d)$ distance between the densities $f_t$ and $f_\infty$, whose analysis is inspired by the proof of \cite[Theorem 2.3]{CJL17}, and then move to the convergence in relative entropy, for which we generalize the proof of \cite[Proposition 3.1]{CGP20}. 

\begin{proposition} \label{pro:convergenceL2_torus}
Let \cref{as:torus} be satisfied. Then, it holds
\begin{equation}
\norm{f_t - f_\infty}_{L^2(\mathbb T^d)} \le \norm{f_0 - f_\infty}_{L^2(\mathbb T^d)} e^{- \zeta t},
\end{equation}
where 
\begin{equation} \label{eq:zeta_def}
\zeta = \frac{\pi^2}{4\beta} - \beta \left( \norm{\nabla V}_{L^\infty(\T^d)} + (1 + \Gamma) \norm{\nabla W}_{L^\infty(\T^d)} \right)^2,
\end{equation}
and $\Gamma$ is defined in equation \eqref{eq:Gamma_def}.
\end{proposition}
\begin{proof}
We rewrite the McKean--Vlasov PDE~\eqref{eq:McKeanVlasovPDE} for the difference $f_t - f_\infty$, using the fact that $f_{\infty}$ solves the stationary Fokker--Planck equation. We then multiply by $f_t - f_\infty$  and integrate by parts to obtain
\begin{equation} \label{eq:decomposition_L2distance}
\begin{aligned}
\frac12 \frac{\d}{\d t} \norm{f_t - f_\infty}_{L^2(\T^d)}^2 &= \int_{\T^d} \nabla \cdot \left( (\nabla V + \nabla W * f_t) f_t - (\nabla V + \nabla W * f_\infty) f_\infty \right) (f_t - f_\infty) \\
&\quad + \beta^{-1} \int_{\T^d} \Delta (f_t - f_\infty) (f_t - f_\infty) \\
&= - \int_{\T^d} \left( (\nabla V + \nabla W * f_t) f_t - (\nabla V + \nabla W * f_\infty) f_\infty \right) \cdot \nabla (f_t - f_\infty) \\
&\quad - \beta^{-1} \norm{\nabla (f_t - f_\infty)}_{L^2(\T^d)}^2 \\
&\eqdef I_t - \beta^{-1} \norm{\nabla (f_t - f_\infty)}_{L^2(\T^d)}^2.
\end{aligned}
\end{equation}
The quantity $I_t$ can then be decomposed as
\begin{equation}
I_t = - \int_{\T^d} (\nabla V + \nabla W * f_t) (f_t - f_\infty) \cdot \nabla (f_t - f_\infty) - \int_{\T^d} \nabla W * (f_t - f_\infty) \cdot \nabla (f_t - f_\infty) f_\infty,
\end{equation}
which, due to Hölder's and Young's convolution inequalities, implies
\begin{equation}
\begin{aligned}
\abs{I_t} &\le \left( \norm{\nabla V}_{L^\infty(\T^d)} + \norm{\nabla W}_{L^\infty(\T^d)} \right) \norm{f_t - f_\infty}_{L^2(\T^d)} \norm{\nabla (f_t - f_\infty)}_{L^2(\T^d)} \\
&\quad + \norm{\nabla W}_{L^\infty(\T^d)} \norm{f_\infty}_{L^\infty(\T^d)} \norm{f_t - f_\infty}_{L^1(\T^d)} \norm{\nabla (f_t - f_\infty)}_{L^1(\T^d)}.
\end{aligned}
\end{equation}
Since the $L^1(\T^d)$ norm can be bounded by the $L^2(\T^d)$ norm, we have
\begin{equation}
\abs{I_t} \le \left( \norm{\nabla V}_{L^\infty(\T^d)} + (1 + \Gamma) \norm{\nabla W}_{L^\infty(\T^d)} \right) \norm{f_t - f_\infty}_{L^2(\T^d)} \norm{\nabla (f_t - f_\infty)}_{L^2(\T^d)},
\end{equation}
where $\Gamma$ is defined in equation \eqref{eq:Gamma_def}, and which, by Young's inequality with $\delta > 0$, gives
\begin{equation} \label{eq:boundIt_L2convergence}
\begin{aligned}
\abs{I_t} &\le \frac\delta2 \left( \norm{\nabla V}_{L^\infty(\T^d)} + (1 + \Gamma) \norm{\nabla W}_{L^\infty(\T^d)} \right)^2 \norm{f_t - f_\infty}_{L^2(\T^d)}^2 \\
&\quad + \frac1{2\delta} \norm{\nabla (f_t - f_\infty)}_{L^2(\T^d)}^2.
\end{aligned}
\end{equation}
Using the bound \eqref{eq:boundIt_L2convergence} in equation \eqref{eq:decomposition_L2distance} and setting $\delta = \beta$, we obtain
\begin{equation}
\begin{aligned}
\frac{\d}{\d t} \norm{f_t - f_\infty}_{L^2(\T^d)}^2 &\le - \beta^{-1} \norm{\nabla (f_t - f_\infty)}_{L^2(\T^d)}^2 \\
&\quad + \beta \left( \norm{\nabla V}_{L^\infty(\T^d)} + (1 + \Gamma) \norm{\nabla W}_{L^\infty(\T^d)} \right)^2 \norm{f_t - f_\infty}_{L^2(\T^d)}^2,
\end{aligned}
\end{equation}
and, since $f_t - f_\infty$ has zero mean, applying Poincaré inequality with the optimal constant $4/\pi^2$ implies
\begin{equation}
\frac{\d}{\d t} \norm{f_t - f_\infty}_{L^2(\T^d)}^2 \le - \zeta \norm{f_t - f_\infty}_{L^2(\T^d)}^2,
\end{equation}
where  $\zeta$ is defined in equation \eqref{eq:zeta_def}. A proof of the Poincaré inequality on the torus in one dimension can be found, e.g., in \cite[Proposition 4.5.5]{BGL14}, while the result in the multidimensional setting follows from the tensorization property. Finally, applying Grönwall's inequality completes the proof.
\end{proof}

\begin{proposition} \label{pro:convergenceH_torus}
Let \cref{as:torus} be satisfied. Then, it holds
\begin{equation}
\mathcal H(f_t | f_\infty) \le \mathcal H(f_0 | f_\infty) e^{- \eta t},
\end{equation}
where 
\begin{equation} \label{eq:eta_def}
\eta = 8 \left( \frac{\pi^2}{\beta\Gamma} - \norm{\Delta W}_{L^\infty(\T^d)} - \beta \norm{\nabla W}_{L^\infty(\T^d)} (\norm{\nabla V}_{L^\infty(\T^d)} + \norm{\nabla W}_{L^\infty(\T^d)}) \right),
\end{equation}
and $\Gamma$ is defined in equation \eqref{eq:Gamma_def}. 
\end{proposition}
\begin{proof}
By differentiating the relative entropy and since $f_t$ solves the McKean--Vlasov PDE \eqref{eq:McKeanVlasovPDE}, we have
\begin{equation}
\begin{aligned}
\frac{\d}{\d t} \mathcal H(f_t | f_\infty) &= \int_{\T^d} \frac{\partial}{\partial t} f_t \left(\log f_t - \log f_\infty \right) + \frac{\d}{\d t}  \int_{\T^d} f_t  \\
&= \int_{\T^d} \left( \nabla \cdot (\nabla V f_t + (\nabla W * f_t) f_t) + \beta^{-1} \Delta f_t \right) \left(\log f_t - \log f_\infty \right),
\end{aligned}
\end{equation}
which, integrating by parts, gives
\begin{equation}
\begin{aligned}
\frac{\d}{\d t} \mathcal H(f_t | f_\infty) &= \int_{\T^d} (\Delta V + \Delta W * f_t) f_t + \int_{\T^d} \nabla \log f_\infty \cdot (\nabla V + \nabla W * f_t) f_t \\
&\quad - \beta^{-1} \int_{\T^d} \abs{\nabla \log f_t}^2 f_t + \beta^{-1} \int_{\T^d} \nabla \log f_\infty \cdot \nabla \log f_t f_t.
\end{aligned}
\end{equation}
By adding and subtracting the Fisher information $\beta^{-1} \mathcal I(f_t | f_\infty)$ we then get
\begin{equation}
\begin{aligned}
\frac{\d}{\d t} \mathcal H(f_t | f_\infty) &= \int_{\T^d} (\Delta V + \Delta W * f_t) f_t + \int_{\T^d} \nabla \log f_\infty \cdot (\nabla V + \nabla W * f_t) f_t \\
&\quad - \beta^{-1} \mathcal I(f_t | f_\infty) + \beta^{-1} \int_{\T^d} \abs{\nabla \log f_\infty}^2 f_t - \beta^{-1} \int_{\T^d} \nabla \log f_\infty \cdot \nabla \log f_t f_t,
\end{aligned}
\end{equation}
which, integrating by parts and since $f_\infty$ solves the integral equation \eqref{eq:integral_equation}, implies
\begin{equation} \label{eq:entropy_decomposition_torus}
\begin{aligned}
\frac{\d}{\d t} \mathcal H(f_t | f_\infty) &= - \beta^{-1} \mathcal I(f_t | f_\infty) + \int_{\T^d} (\Delta V + \Delta W * f_t) f_t - \int_{\T^d} (\Delta V + \Delta W * f_\infty) f_t  \\
&\quad + \int_{\T^d} \nabla \log f_\infty \cdot (\nabla V + \nabla W * f_t) f_t - \int_{\T^d} \nabla \log f_\infty \cdot (\nabla V + \nabla W * f_\infty) f_t \\
&= - \beta^{-1} \mathcal I(f_t | f_\infty) + \int_{\T^d} \Delta W * (f_t - f_\infty) f_t + \int_{\T^d} \nabla \log f_\infty \cdot \nabla W * (f_t - f_\infty) f_t.
\end{aligned}
\end{equation}
Let us notice that, due to an integration by parts, it holds
\begin{equation}
\int_{\T^d} \Delta W * (f_t - f_\infty) f_\infty + \int_{\T^d} \nabla \log f_\infty \cdot \nabla W * (f_t - f_\infty) f_\infty = 0,
\end{equation}
and therefore adding this quantity to the right-hand side of equation \eqref{eq:entropy_decomposition_torus} we obtain
\begin{equation}
\begin{aligned}
\frac{\d}{\d t} \mathcal H(f_t | f_\infty) &= - \beta^{-1} \mathcal I(f_t | f_\infty) + \int_{\T^d} \Delta W * (f_t - f_\infty) (f_t - f_\infty) \\
&\quad + \int_{\T^d} \nabla \log f_\infty \cdot \nabla W * (f_t - f_\infty) (f_t - f_\infty) \\
&\eqdef - \beta^{-1} \mathcal I(f_t | f_\infty) + I^{(1)}_t + I^{(2)}_t.
\end{aligned}
\end{equation}
We now need to bound the terms in the right-hand side. First, by Hölder's and Young's convolution inequalities and since $\nabla \log f_\infty = - \beta (\nabla V + \nabla W * f_\infty)$, we have
\begin{equation}
\begin{aligned}
\abs{I^{(1)}_t} &\le \norm{\Delta W}_{L^\infty(\T^d)} \norm{f_t - f_\infty}_{L^1(\T^d)}^2, \\
\abs{I^{(2)}_t} &\le \beta \norm{\nabla W}_{L^\infty(\T^d)} (\norm{\nabla V}_{L^\infty(\T^d)} + \norm{\nabla W}_{L^\infty(\T^d)}) \norm{f_t - f_\infty}_{L^1(\T^d)}^2,
\end{aligned}
\end{equation}
and then, due to equation \eqref{eq:TV_L1} and the Csiszár--Kullback--Pinsker (CKP) inequality (see, e.g., \cite{BoV05}), we get
\begin{equation} \label{eq:CKP_inequality}
\norm{f_t - f_\infty}_{L^1(\T^d)}^2 \le 8 \mathcal H(f_t | f_\infty).
\end{equation}
Finally, \cref{lem:LSI_invariant_torus} yields
\begin{equation}
\mathcal I(f_t | f_\infty) \ge \frac{8\pi^2}{\Gamma} \mathcal H(f_t | f_\infty),
\end{equation}
which implies
\begin{equation}
\frac{\d}{\d t} \mathcal H(f_t | f_\infty) \le - \eta \mathcal H(f_t | f_\infty),
\end{equation}
where $\eta$ is defined in equation \eqref{eq:eta_def}. The desired result is then obtained applying Grönwall's inequality.
\end{proof}

\begin{remark}
In \cref{pro:convergenceL2_torus,pro:convergenceH_torus}, the exponents $\zeta$ and $\eta$ are always strictly positive in the high temperature regime, i.e., when $\beta$ is sufficiently small. Moreover, from the proof of~\cite[Proposition 3.1]{CDP20} it follows that we can replace $\norm{\Delta W}_{L^\infty(\T^d)}$ in~\eqref{eq:eta_def} with $\norm{\Delta W_u}_{L^\infty(\T^d)}$, where $W_u$ denotes the non-$H$-stable part of $W$ and is defined in~\cite[Definition 2.1]{CDP20}. However, the same argument cannot be applied to $\norm{\nabla W}_{L^\infty(\T^d)}$.
\end{remark}

The corollary below is a straightforward consequence of \cref{pro:convergenceL2_torus,pro:convergenceH_torus}, and quantifies the rate of convergence in $L^1(\T^d)$.

\begin{corollary} \label{cor:convergenceL1_torus}
Let \cref{as:torus} be satisfied. Then, it holds for $i = 1, 2$
\begin{equation}
\norm{f_t - f_\infty}_{L^1(\T^d)} \le c_i e^{- a_i t},
\end{equation}
with $a_1 = \zeta$, $a_2 = \eta/2$, where $\zeta,\eta$ are defined in  \eqref{eq:zeta_def} and \eqref{eq:eta_def}, respectively, and
\begin{equation}
c_1 = \norm{f_0 - f_\infty}_{L^2(\T^d)}, \qquad c_2 = 2\sqrt{2 \mathcal H(f_0|f_\infty)}.
\end{equation}
\end{corollary}
\begin{proof}
The desired result follows applying \cref{pro:convergenceL2_torus,pro:convergenceH_torus}, the CKP inequality \eqref{eq:CKP_inequality}, and bounding the $L^1(\T^d)$ norm with the $L^2(\T^d)$ one.
\end{proof}

\subsubsection{Convergence of the linearized process}

Similarly to the analysis for the whole space, we first need to prove that the laws of the processes satisfy an LSI uniformly in time. In the lemma below, whose argument is inspired from the proof of \cite[Lemma 5.1]{LaL23}, we show that this holds true as long as the density of the initial condition is bounded away from zero. 

\begin{lemma} \label{lem:LSI_torus}
Let \cref{as:torus} be satisfied, and assume that there exists a constant $\kappa > 1$ such that for all $x \in \T^d$
\begin{equation} \label{eq:bound_0_torus}
\frac1\kappa \le f_0(x) \le \kappa \qquad \text{and} \qquad \frac1\kappa \le g_0(x) \le \kappa,
\end{equation}
where $f_0$ and $g_0$ are the densities with respect to the Lebesgue measure of the initial distributions $\mu_0$ and $\nu_0$. Then, $\nu_t$ satisfies a time-uniform LSI with constant
\begin{equation}
\Xi = \frac{\kappa\Gamma^2}{2\pi^2},
\end{equation}
where $\Gamma$ is defined in \eqref{eq:Gamma_def}. Moreover, if $a_i > 0$ and $C_i/a_i < \mathbb W(1)$ for either $i=1$ or $i=2$, with $a_i$ and $C_i$ given in \cref{cor:convergenceL1_torus} and equation \eqref{eq:Ci_def} below, and where $\mathbb W$ denotes the Lambert function, then $\mu_t$ satisfies a time-uniform LSI with constant
\begin{equation}
\widetilde \Xi_i = \frac{\kappa\Gamma^2}{2\pi^2(1 - (C_i/a_i)e^{C_i/a_i})}.
\end{equation}
\end{lemma}
\begin{proof}
Let $\varphi_t$ and $\psi_t$ be the Radon--Nikodym derivative of $\nu_t$ and $\mu_t$ with respect to $\mu_\infty = \nu_\infty$, i.e.,
\begin{equation}
\varphi_t = \frac{\d \nu_t}{\d \mu_\infty} = \frac{g_t}{f_\infty} \qquad \text{and} \qquad \psi_t = \frac{\d \mu_t}{\d \mu_\infty} = \frac{f_t}{f_\infty},
\end{equation}
and note that, since $f_t,g_t$ and $f_\infty$ solve the time-dependent and stationary Fokker--Planck equations \eqref{eq:FP_equations}, respectively, $\varphi_t$ solves
\begin{equation} \label{eq:phi_evolution}
\frac{\partial \varphi_t}{\partial t} = \beta^{-1} \Delta \varphi_t - (\nabla V + \nabla W * f_\infty) \cdot \nabla \varphi_t,
\end{equation}
and $\psi_t$ satisfies
\begin{equation} \label{eq:psi_evolution}
\begin{aligned}
\frac{\partial \psi_t}{\partial t} &= \beta^{-1} \Delta \psi_t + \Delta W * ((\psi_t - 1) f_\infty) \psi_t - \beta (\nabla V + \nabla W * f_\infty) \cdot \nabla W * ((\psi_t - 1) f_\infty) \psi_t \\
&\quad - (\nabla V + \nabla W * f_\infty) \cdot \nabla \psi_t + \nabla W * ((\psi_t - 1) f_\infty) \cdot \nabla \psi_t.
\end{aligned}
\end{equation}
We now follow the proof of~\cite[Lemma 5.1]{LaL23} that is based on the argument presented in~\cite[Theorem 2]{GLM24}; see also the proof of~\cite[Theorem 5]{NPR24}. Fix $T > 0$ and $x \in \T^d$, and let $(Z_t^\nu)_{t\in[0,T]}$ be the solution of the SDE
\begin{equation} \label{eq:Z_nu_SDE}
\d Z_t^\nu = - (\nabla V(Z_t^\nu) + \nabla W * f_\infty (Z_t^\nu)) \dd t + \sqrt{2\beta^{-1}} \dd B_t,
\end{equation}
with initial condition $Z_0^\nu = x$. Using equation \eqref{eq:phi_evolution} and applying Itô lemma to the function $\varphi_{T-t}$, we obtain
\begin{equation}
\d \varphi_{T-t}(Z_t^\nu) = \sqrt{2\beta^{-1}} \nabla \varphi_{T-t}(Z_t^\nu) \cdot \d B_t,
\end{equation}
which, integrating, taking the expectation, and noting that the Itô integral has expectation zero, gives
\begin{equation}
\E [\varphi_{T-t}(Z_t^\nu)] = \varphi_T(x).
\end{equation}
Then, due to hypothesis \eqref{eq:bound_0_torus}, choosing $t=T$ yields for all $T > 0$
\begin{equation} \label{eq:bound_density_nu}
\frac1{\kappa \Gamma} \le \varphi_T(x) \le \kappa \Gamma.
\end{equation}
The argument is similar, but more complex, for the measure $\mu_t$. Let $(Z_t^\mu)_{t\in[0,T]}$ be the solution of the SDE
\begin{equation} \label{eq:Z_mu_SDE}
\d Z_t^\mu = - (\nabla V(Z_t^\mu) + \nabla W * f_\infty (Z_t^\mu)) \dd t + \nabla W * ((\psi_{T-t} - 1) f_\infty)(Z_t^\mu) \dd t + \sqrt{2\beta^{-1}} \dd B_t,
\end{equation}
with initial condition $Z_0^\mu = x$. Using equation \eqref{eq:psi_evolution} and applying Itô lemma to the function $\psi_{T-t}$, we obtain
\begin{equation}
\begin{aligned}
\d \psi_{T-t}(Z_t^\mu) &= \beta (\nabla V(Z_t^\mu) + \nabla W * f_\infty(Z_t^\mu)) \cdot \nabla W * ((\psi_{T-t} - 1) f_\infty) (Z_t^\mu) \psi_{T-t}(Z_t^\mu) \dd t \\
&\quad - \Delta W * ((\psi_{T-t} - 1) f_\infty)(Z_t^\mu) \psi_{T-t}(Z_t^\mu) \dd t + \sqrt{2\beta^{-1}} \nabla \psi_{T-t}(Z_t^\mu) \cdot \d B_t,
\end{aligned}
\end{equation}
which, integrating and taking the expectation, gives
\begin{equation} \label{eq:expectation_psi_Tt}
\begin{aligned}
\E [\psi_{T-t}(Z_t^\mu)] &= \psi_T(x) - \int_0^t \E \left[ \Delta W * ((\psi_{T-s} - 1) f_\infty)(Z_s^\mu) \psi_{T-s}(Z_s^\mu) \dd s \right] \\
&\hspace{-0.5cm} + \beta \int_0^t \E \left[ (\nabla V(Z_s^\mu) + \nabla W * f_\infty(Z_s^\mu)) \cdot \nabla W * ((\psi_{T-s} - 1) f_\infty) (Z_s^\mu) \psi_{T-s}(Z_s^\mu) \dd s \right].
\end{aligned}
\end{equation}
Using \cref{cor:convergenceL1_torus} we have
\begin{equation} \label{eq:bound_pros}
\norm{(\psi_{T-s} - 1) f_\infty}_{L^1(\T^d)} \le c_i e^{-a_i(T-s)},
\end{equation}
which implies
\begin{equation}
\E [\psi_{T-t}(Z_t^\mu)] \le \psi_T(x) + C_i \int_0^t e^{-a_i(T-s)} \E [\psi_{T-s}(Z_s^\mu)] \dd s,
\end{equation}
where
\begin{equation} \label{eq:Ci_def}
C_i = c_i \left( \norm{\Delta W}_{L^\infty(\T^d)} + \beta \norm{\nabla W}_{L^\infty(\T^d)} (\norm{\nabla V}_{L^\infty(\T^d)} + \norm{\nabla W}_{L^\infty(\T^d)}) \right).
\end{equation}
By Grönwall's inequality we deduce
\begin{equation} \label{eq:bound_psi_Tt}
\E [\psi_{T-t}(Z_t^\mu)] \le \psi_T(x) e^{C_i \int_0^t e^{-a_i(T-s)} \dd s} \le \psi_T(x) e^{C_i/a_i},
\end{equation}
which, choosing $t=T$ and by hypothesis \eqref{eq:bound_0_torus}, yields 
\begin{equation} \label{eq:psiT_lower}
\psi_T(x) \ge \frac1{\kappa \Gamma} e^{- C_i/a_i}.
\end{equation}
Moreover, using again equations \eqref{eq:expectation_psi_Tt} and \eqref{eq:bound_pros} with $t=T$, we have
\begin{equation}
\psi_T(x) \le \E [\psi_0(Z_t^\mu)] + C_i \int_0^t e^{-a_i(T-s)} \E [\psi_{T-s}(Z_s^\mu)] \dd s,
\end{equation}
which, by equation \eqref{eq:bound_psi_Tt}, gives
\begin{equation} \label{eq:psiT_upper}
\psi_T(x) \le \E [\psi_0(Z_t^\mu)] + C_i \psi_T(x) e^{C_i/a_i} \int_0^t e^{-a_i(T-s)} \dd s \le \kappa \Gamma + \frac{C_i}{a_i} \psi_T(x) e^{C_i/a_i}.
\end{equation}
Combining bounds \eqref{eq:psiT_lower} and \eqref{eq:psiT_upper}, we obtain for all $T > 0$
\begin{equation} \label{eq:bound_density_mu}
\frac1{\kappa \Gamma e^{C_i/a_i}} \le \psi_T(x) \le \frac{\kappa\Gamma}{1 - (C_i/a_i)e^{C_i/a_i}},
\end{equation}
where we remark that, due to the assumptions on $a_i$ and $C_i$, the inequalities make sense since the left-hand side and the right-hand side are always smaller and greater than one, respectively. Notice also that it holds
\begin{equation}
\kappa \Gamma e^{C_i/a_i} \le \frac{\kappa\Gamma}{1 - (C_i/a_i)e^{C_i/a_i}}.
\end{equation}
Therefore, the desired results follow from the bounds \eqref{eq:bound_density_nu} and \eqref{eq:bound_density_mu} by applying the Holley--Stroock perturbation lemma (see, e.g., \cite[Lemma 2.3]{GLM24}) and \cref{lem:LSI_invariant_torus}.
\end{proof}

Using the LSI, we can finally prove that the measures $\mu_t$ and $\nu_t$ are exponentially close in time, in relative entropy. The argument is similar to the proof of \cref{thm:entropy_whole}.

\begin{theorem} \label{thm:entropy_torus}
Let the assumptions of \cref{lem:LSI_torus} be satisfied. Then, it holds for $i=1,2$
\begin{equation}
\mathcal H (\mu_t | \nu_t) \le \sigma_\Xi(t) \defeq \begin{cases}
\left( \mathcal H (\mu_0 | \nu_0) + \frac{\beta^2 c_i \Xi \norm{\nabla W}_{L^\infty(\T^d)}^2}{2(2 - \beta a_i \Xi)} \right) e^{-a_i t}, & \text{if } \Xi < \frac2{a_i\beta}, \\
\left( \mathcal H (\mu_0 | \nu_0) + \frac12 \beta c_i \norm{\nabla W}_{L^\infty(\T^d)}^2 t \right) e^{- \frac2{\beta\Xi} t}, & \text{if } \Xi = \frac2{a_i\beta}, \\
\left( \mathcal H (\mu_0 | \nu_0) + \frac{\beta^2 c_i \Xi \norm{\nabla W}_{L^\infty(\T^d)}^2}{2(\beta a_i \Xi - 2)} \right) e^{- \frac2{\beta\Xi} t}, & \text{if } \Xi > \frac2{a_i\beta},
\end{cases}
\end{equation}
where $a_i$ and $c_i$ are defined in \cref{cor:convergenceL1_torus}.
\end{theorem}
\begin{proof}
By \cite[Lemma 3.1]{LaL23} we have the estimate
\begin{equation}
\mathcal H (\mu_t | \nu_t) + \frac1{2\beta} \int_0^t  \mathcal I(\mu_s | \nu_s) \dd s \le \mathcal H (\mu_0 | \nu_0) + \frac\beta2 \int_0^t \int_{\T^d} \abs{(\nabla W * (f_s - f_\infty))(x)}^2 f_s(x) \dd x \dd s,
\end{equation}
which, due to \cref{lem:LSI_torus}, implies
\begin{equation} \label{eq:entropy_initial_estimate_torus}
\mathcal H (\mu_t | \nu_t) + \frac2{\beta\Xi} \int_0^t  \mathcal H(\mu_s | \nu_s) \dd s \le \mathcal H (\mu_0 | \nu_0) + \frac\beta2 \int_0^t \int_{\T^d} \abs{(\nabla W * (f_s - f_\infty))(x)}^2 f_s(x) \dd x \dd s.
\end{equation}
Then, \cref{cor:convergenceL1_torus} gives
\begin{equation}
\abs{(\nabla W * (f_s - f_\infty))(x)}^2 \le \norm{\nabla W}_{L^\infty(\T^d)}^2 \norm{f_s - f_\infty}_{L^1(\T^d)}^2 \le c_i \norm{\nabla W}_{L^\infty(\T^d)}^2 e^{-a_i s},
\end{equation}
which yields
\begin{equation}
\mathcal H (\mu_t | \nu_t) \le \mathcal H (\mu_0 | \nu_0) + \frac{\beta c_i \norm{\nabla W}_{L^\infty(\T^d)}^2}2 \int_0^t e^{- a_i s} \dd s - \frac2{\beta\Xi} \int_0^t  \mathcal H(\mu_s | \nu_s) \dd s.
\end{equation}
Applying the Grönwall's inequality in \cite[Lemma A.1]{LaL23}, we have
\begin{equation}
\begin{aligned}
\mathcal H (\mu_t | \nu_t) &\le \mathcal H (\mu_0 | \nu_0) e^{- \frac2{\beta\Xi} t} + \frac{\beta c_i \norm{\nabla W}_{L^\infty(\T^d)}^2}2 \int_0^t e^{- \frac2{\beta\Xi} (t-s)} e^{- a_i s} \dd s, \\
&= \mathcal H (\mu_0 | \nu_0) e^{- \frac2{\beta\Xi} t} + \frac{\beta c_i \norm{\nabla W}_{L^\infty(\T^d)}^2}2 e^{- \frac2{\beta\Xi} t} \int_0^t e^{\frac{2 - a_i\beta\Xi}{\beta\Xi}s} \dd s.
\end{aligned}
\end{equation}
Therefore, if $\Xi = 2/(a_i \beta)$, we deduce
\begin{equation}
\mathcal H (\mu_t | \nu_t) \le \left( \mathcal H (\mu_0 | \nu_0) + \frac{\beta c_i \norm{\nabla W}_{L^\infty(\T^d)}^2}2 t \right) e^{- \frac2{\beta\Xi} t},
\end{equation}
otherwise we get
\begin{equation}
\mathcal H (\mu_t | \nu_t) \le \mathcal H (\mu_0 | \nu_0) e^{- \frac2{\beta\Xi} t} + \frac{\beta^2 c_i \Xi \norm{\nabla W}_{L^\infty(\T^d)}^2}{2(2 - \beta a_i \Xi)} \left( e^{-a_i t} - e^{- \frac2{\beta\Xi} t} \right),
\end{equation}
which concludes the proof.
\end{proof}

Similarly to the case where the state space is the whole space, the following convergence in Wasserstein distance is a direct consequence of the Talagrand’s transportation inequality.

\begin{corollary} \label{cor:wasserstein_torus}
Let the assumptions of \cref{thm:entropy_torus} be satisfied. Then, it holds
\begin{equation}
\mathcal W_2 (\mu_t, \nu_t) \le \sqrt{ \Xi \sigma_\Xi(t)},
\end{equation}
where $\mathcal W_2$ denotes the 2-Wasserstein distance and $\sigma_\Xi(t)$ is defined in \cref{thm:entropy_torus}.
\end{corollary}
\begin{proof}
The desired result follows from \cref{thm:entropy_torus} using \cite[Theorem 1]{OtV00}.
\end{proof}

\section{Applications} \label{sec:applications}

In this section, we present two applications in which the linearization procedure can be used. From the analysis presented in the previous section, the law of the solution of the linearized McKean SDE $\nu_t$ is close to the law of the nonlinear process $\mu_t$, and their distance vanishes exponentially fast as $t \to \infty$. Therefore, the linear process $Y_t$ can provide a reliable approximation when studying properties that depend on the long-time behaviour of the nonlinear dynamics $X_t$. We first study the maximum likelihood estimator for the mean field dynamics. Then, we compute the diffusion coefficient of the limiting Brownian motion in the diffusive-mean field limit of weakly interacting diffusions in periodic potentials.

\subsection{Maximum likelihood estimation} \label{sec:MLE}

We consider the case where the confining and interaction potentials $V$ and $W$ depend on unknown parameters. Our goal is to infer these parameters by observing a trajectory of the mean field SDE \eqref{eq:McKeanSDE}. Hence, consider the equation on the whole space $\R^d$ and in the time horizon $[0,T]$
\begin{equation} \label{eq:McKeanSDE_MLE}
\d X_t = - \nabla V(X_t; \theta) \dd t - (\nabla W(\cdot; \theta) * f_t(\cdot; \theta)) (X_t) \dd t + \sqrt{2 \beta^{-1}} \dd B_t,
\end{equation}
where $\theta \in \Theta$ is the vector of unknown parameters and $\Theta$ is the admissible set. The noise strength is assumed to be known. We denote by $\theta_0$ the exact value and follow a maximum likelihood approach~\cite{Kas90, Bis11, DeH23}. In particular, by the Girsanov theorem, the log-likelihood is
\begin{equation} \label{eq:likelihood}
\begin{aligned}
\mathfrak L_T(\theta) &= - \frac\beta2 \int_0^T \left( \nabla V(X_t; \theta) + (\nabla W(\cdot; \theta) * f_t(\cdot; \theta)) (X_t) \right)^\top \dd X_t \\
&\quad - \frac\beta4 \int_0^T \abs{\nabla V(X_t; \theta) + (\nabla W(\cdot; \theta) * f_t(\cdot; \theta)) (X_t)}^2 \dd t,
\end{aligned}
\end{equation}
and the MLE is defined as
\begin{equation}
\widehat\theta_T = \argmax_{\theta\in\Theta} \mathfrak L_T(\theta).
\end{equation}
We will use the following notation. 
\begin{equation} \label{eq:def_b}
\begin{aligned}
b_\theta(t,x) &\defeq - \nabla V(x; \theta) - (\nabla W(\cdot; \theta) * f_t(\cdot; \theta)) (x), \\
\widetilde b_\theta(x) &\defeq - \nabla V(x; \theta) - (\nabla W(\cdot; \theta) * f_\infty(\cdot; \theta)) (x),
\end{aligned}
\end{equation}
and notice that maximizing $\mathfrak L_T$ is equivalent to maximize the quantity
\begin{equation}\label{e:mlf_alt}
\ell_T(\theta) = \frac\beta2 \int_0^T \left( b_\theta(t,X_t) - b_{\theta_0}(t,X_t) \right)^\top \dd B_t - \frac\beta4 \int_0^T \abs{b_\theta(t,X_t) - b_{\theta_0}(t,X_t)}^2 \dd t.
\end{equation}
In particular, the MLE is defined as
\begin{equation}
\widehat\theta_T = \argmax_{\theta\in\Theta} \mathfrak L_T(\theta) = \argmax_{\theta\in\Theta} \ell_T(\theta).
\end{equation}
We now define the linearized version of the log-likelihood around the invariant measure $f_\infty(\cdot; \theta)$,
\begin{equation} \label{eq:likelihood_linearized}
\begin{aligned}
\widetilde{\mathfrak L}_T(\theta) &= - \frac\beta2 \int_0^T \left( \nabla V(X_t; \theta) + (\nabla W(\cdot; \theta) * f_\infty(\cdot; \theta)) (X_t) \right)^\top \dd X_t \\
&\quad - \frac\beta4 \int_0^T \abs{\nabla V(X_t; \theta) + (\nabla W(\cdot; \theta) * f_\infty(\cdot; \theta)) (X_t)}^2 \dd t,
\end{aligned}
\end{equation}
and the analogue of~\eqref{e:mlf_alt}
\begin{equation} \label{eq:equivalent_likelihood_tilde}
\widetilde\ell_T(\theta) \defeq \frac\beta2 \int_0^T \left( \widetilde b_\theta(X_t) - b_{\theta_0}(t,X_t) \right)^\top \dd B_t - \frac\beta4 \int_0^T \abs{\widetilde b_\theta(X_t) - b_{\theta_0}(t,X_t)}^2 \dd t.
\end{equation}
The linearized MLE is therefore given by
\begin{equation}
\widetilde\theta_T = \argmax_{\theta\in\Theta} \widetilde{\mathfrak L}_T(\theta) = \argmax_{\theta\in\Theta} \widetilde\ell_T(\theta).
\end{equation}
It is already known that the MLE $\widehat \theta_T$, which is computed on a path from the solution to the McKean SDE, is asymptotically unbiased \cite[Theorem 3.1]{LiQ22}, meaning that 
\begin{equation}
\lim_{T \to \infty} \widehat\theta_T = \theta_0 \qquad \text{a.s.},
\end{equation}
and we aim to show that the same result holds true for the linearized MLE $\widetilde\theta_T$.

\subsubsection{Asymptotic unbiasedness of the linearized MLE}

In this section, we show that replacing the density of the process by its invariant measure in the calculation of the likelihood function, while observing a path of the McKean SDE, does not affect the resulting MLE. We emphasize the fact that we keep the same observation model, i.e., that the ``linearized'' likelihood is still evaluated at a path of the nonlinear process. We will make the following assumptions.

\begin{assumption} \label{as:MLE}
The following properties hold true.
\begin{enumerate}[leftmargin=*]
\item The admissible set $\Theta$ is compact.
\item\label{it:Lipschitz} The gradients of the confining and interaction potentials are Lipschitz in $\theta$, i.e., there exist $L_V, L_W \colon \R^d \to \R$, $L_V, L_W > 0$, such that for all $\theta_1, \theta_2 \in \Theta$ and for all $x \in \R^d$
\begin{equation}
\begin{aligned}
\abs{\nabla V(x; \theta_1) - \nabla V(x; \theta_2)} &\le L_V(x) \abs{\theta_1 - \theta_2}, \\
\abs{\nabla W(x; \theta_1) - \nabla W(x; \theta_2)} &\le L_W(x) \abs{\theta_1 - \theta_2},
\end{aligned}
\end{equation}
where $L_V(x)$ and $L_W(x)$ are polynomially bounded.
\item\label{it:identifiability} The function
\begin{equation}
\theta \mapsto \int_{\R^d} \abs{\widetilde b_\theta(x) - \widetilde b_{\theta_0}(x)}^2 f_\infty(x;\theta_0) \dd x 
\end{equation}
admits only one zero at $\theta = \theta_0$.
\end{enumerate}
\end{assumption}

\begin{remark}
The condition in \cref{as:MLE}\ref{it:identifiability}, which appears also in \cite{Van01}, is necessary to guarantee the identifiability of the unknown parameter $\theta_0$, and it is related to the nondegeneracy of the Fisher information matrix in \cite{DeH23}. We notice that this condition is satisfied, e.g., if $\widetilde b_\theta$ is an affine function of $\theta$.
\end{remark}

The main result of this section is the asymptotic unbiasedness of $\widetilde \theta_T$, when evaluated at a path of the McKean SDE.

\begin{theorem}\label{thm:MLE_lin}
Under \cref{as:whole,as:MLE}, it holds
\begin{equation}
\lim_{T \to \infty} \widetilde\theta_T = \theta_0 \qquad \text{a.s.}
\end{equation}
\end{theorem}

\begin{remark}
We reiterate that our observation model is that of a single (nonstationary) path of the McKean SDE. Ideally, we would like to consider the case where we observe one or more paths of the interacting particle system. To show asymptotic unbiasedness of the linearized MLE in this setting, we would need to approximate the path of the nonlinear process by a path of the particle system, using the uniform propagation of chaos property. We will omit the details. Needless to say, in the numerical experiments we observe a \emph{single} trajectory of the particle system.
\end{remark}

Our theoretical analysis is inspired by the proof of \cite[Corollary 4.1]{BoB83}, where the linear case is treated. In fact, in our approach the likelihood function is computed from the linear process, and the nonlinearity only appears in the observation model. The proof of~\cref{thm:MLE_lin} relies on the following preliminary results.

\begin{lemma} \label{lem:convergence_likelihood2}
Under \cref{as:whole,as:MLE}, it holds for all $\theta\in\Theta$
\begin{equation}
\lim_{T\to\infty} \frac1T \int_0^T \abs{\widetilde b_\theta(X_t) - b_{\theta_0}(t,X_t)}^2 \dd t = \int_{\R^d} \abs{\widetilde b_\theta(x) - \widetilde b_{\theta_0}(x)}^2 f_\infty(x;\theta_0) \dd x \qquad \text{a.s.}
\end{equation}
\end{lemma}
\begin{proof}
Let us denote by $I_T$ the quantity in the left-hand side of the result, and notice that the following decomposition holds true
\begin{equation}
\begin{aligned}
I_T &= \frac1T \int_0^T \abs{\widetilde b_\theta(X_t) - b_{\theta_0}(t,X_t)}^2 \dd t \\
&=  \frac1T \int_0^T \abs{\widetilde b_\theta(X_t) - \widetilde b_{\theta_0}(X_t)}^2 \dd t \\
&\quad + \frac1T \int_0^T \abs{\widetilde b_{\theta_0}(X_t) - b_{\theta_0}(t,X_t)}^2 \dd t \\
&\quad + \frac2T \int_0^T \left( \widetilde b_\theta(X_t) - \widetilde b_{\theta_0}(X_t) \right)^\top \left( \widetilde b_{\theta_0}(X_t) - b_{\theta_0}(t,X_t) \right) \dd t \\
&= I_T^{(1)} + I_T^{(2)} + I_T^{(3)}.
\end{aligned}                                                                   
\end{equation}
By the ergodic theorem we first have
\begin{equation}
\lim_{T\to\infty} I_T^{(1)} = \int_{\R^d} \abs{\widetilde b_\theta(x) - \widetilde b_{\theta_0}(x)}^2 f_\infty(x;\theta_0) \dd x \qquad \text{a.s.}
\end{equation}
We then consider $I_T^{(2)}$ and, due to the definition \eqref{eq:def_b} and the Cauchy--Schwarz inequality, we get
\begin{equation}
\begin{aligned}
I_T^{(2)} &\le \frac1T \int_0^T \left( \int_{\R^d} \abs{\nabla W(X_t - y; \theta_0)} \abs{f_t(y;\theta_0) - f_\infty(y; \theta_0)} \dd y \right)^2 \dd t \\
&\le \frac1T \int_0^T \left( \int_{\R^d} \abs{\nabla W(X_t - y; \theta_0)}^2 \left( f_t(y;\theta_0) + f_\infty(y; \theta_0) \right) \dd y \right) \dd t \\
&\hspace{1cm} \times \norm{f_t(\cdot;\theta_0) - f_\infty(\cdot; \theta_0)}_{L^1(\R^d)},
\end{aligned}
\end{equation}
which, by equation \eqref{eq:convergenceL1_density_whole}, implies
\begin{equation}
I_T^{(2)} \le \frac{K}{T} e^{-\frac12 \alpha t} \int_0^T \left( \int_{\R^d} \abs{\nabla W(X_t - y; \theta_0)}^2 \left( f_t(y;\theta_0) + f_\infty(y; \theta_0) \right) \dd y \right) \dd t.
\end{equation}
Then, by the fact that $\nabla W$ is polynomially bounded, the uniform boundedness of the moments, and the ergodic theorem, we deduce that $I_T^{(2)} \to 0$ almost surely. Moreover, by Cauchy--Schwarz inequality we have
\begin{equation}
I_T^{(3)} \le 2 \sqrt{I_T^{(1)} I_T^{(2)}},
\end{equation}
which yields that $I_T^{(3)}$ converges to zero almost surely. The desired result then follows from the convergence of the terms $I_T^{(1)}, I_T^{(2)}, I_T^{(3)}$.
\end{proof}

\begin{corollary} \label{cor:convergence_likelihood2}
Under \cref{as:whole,as:MLE}, it holds
\begin{equation}
\lim_{T\to\infty} \frac1T \int_0^T \left( \widetilde b_{\theta_0}(X_t) - b_{\theta_0}(t,X_t) \right)^2 \dd t = 0 \qquad \text{a.s.}
\end{equation}
\end{corollary}
\begin{proof}
The desired result follows from \cref{lem:convergence_likelihood2} by replacing $\theta$ with $\theta_0$.
\end{proof}

\begin{lemma} \label{lem:convergence_likelihood1}
Under \cref{as:whole,as:MLE}, it holds
\begin{equation}
\lim_{T\to\infty} \sup_{\theta \in \Theta} \frac1T \int_0^T \left( \widetilde b_\theta(X_t) - b_{\theta_0}(t,X_t) \right) \dd B_t = 0 \qquad \text{a.s.}
\end{equation}
\end{lemma}
\begin{proof}
Let us define the function
\begin{equation}
e(T,\theta) =  \frac1T \int_0^T \left( \widetilde b_\theta(X_t) - b_{\theta_0}(t,X_t) \right) \dd B_t,
\end{equation}
which is uniformly continuous in $\theta$ on $\Theta$, uniformly in $T \in [1,\infty)$, with probability $1$, by \cite[Theorem 3.1]{BoB83}. Then, by \cref{lem:convergence_likelihood2} and the strong law of large numbers for martingales we deduce that for all $\theta \in \Theta$
\begin{equation}
\lim_{T \to \infty} e(T, \theta) = 0 \qquad \text{a.s.}
\end{equation}
Finally, the result is obtained following the same argument as in the proof of \cite[Lemma 4.1]{BoB83}.
\end{proof}

\begin{lemma} \label{lem:f_Lipschitz}
Under \cref{as:whole,as:MLE}, there exists a constant $L > 0$, independent of $T$, such that for all $\theta_1, \theta_2 \in \Theta$
\begin{equation}
\abs{\frac1T \int_0^T \abs{\widetilde b_{\theta_1}(X_t) - b_{\theta_0}(t,X_t)}^2 \dd t - \frac1T \int_0^T \abs{\widetilde b_{\theta_2}(X_t) - b_{\theta_0}(t,X_t)}^2 \dd t} \le L \abs{\theta_1 - \theta_2}.
\end{equation}
\end{lemma}
\begin{proof}
In order to show the result, we define the function 
\begin{equation} \label{eq:h_def}
h(T,\theta) = \frac1T \int_0^T \abs{\widetilde b_\theta(X_t) - b_{\theta_0}(t,X_t)}^2 \dd t,
\end{equation}
and show that it is Lipschitz in $\theta$, uniformly in $T$. By the Cauchy--Schwarz inequality, we have for a constant $B>0$
\begin{equation}
\begin{aligned}
&\abs{h(T, \theta_1) - h(T, \theta_2)} = \abs{\frac1T \int_0^T \abs{\widetilde b_{\theta_1}(X_t) - b_{\theta_0}(t,X_t)}^2 \dd t - \frac1T \int_0^T \abs{\widetilde b_{\theta_2}(X_t) - b_{\theta_0}(t,X_t)}^2 \dd t} \\
&\hspace{1.5cm}\le \frac1T \int_0^T \abs{\widetilde b_{\theta_1}(X_t) - \widetilde b_{\theta_2}(X_t)} \abs{\widetilde b_{\theta_1}(X_t) + \widetilde b_{\theta_2}(X_t) - 2b_{\theta_0}(t,X_t)} \dd t \\
&\hspace{1.5cm}\le \left(\frac1T \int_0^T \abs{\widetilde b_{\theta_1}(X_t) - \widetilde b_{\theta_2}(X_t)}^2 \right)^{1/2} \left(\frac1T \int_0^T \abs{\widetilde b_{\theta_1}(X_t) + \widetilde b_{\theta_2}(X_t) - 2b_{\theta_0}(t,X_t)}^2 \right)^{1/2} \\
&\hspace{1.5cm}\le B \left(\frac1T \int_0^T \abs{\widetilde b_{\theta_1}(X_t) - \widetilde b_{\theta_2}(X_t)}^2 \right)^{1/2},
\end{aligned}
\end{equation}
where the uniform bound in time of the last inequality is due to \cref{lem:convergence_likelihood2}. By the definition of $\widetilde b_\theta$ in equation \eqref{eq:def_b} we obtain
\begin{equation}
\begin{aligned}
\abs{h(T, \theta_1) - h(T, \theta_2)}^2 &\le \frac{B}{T} \int_0^T \abs{\nabla V(X_t; \theta_1) - \nabla V(X_t;\theta_2)}^2 \dd t \\
&+ \frac{B}{T} \int_0^T \left( \int_{\R^d} \abs{\nabla W(X_t - y; \theta_1)} \abs{f_\infty(y; \theta_1) - f_\infty(y; \theta_2)} \dd y \right)^2 \dd t \\
&+ \frac{B}{T} \int_0^T \left( \int_{\R^d} \abs{\nabla W(X_t - y; \theta_1) - \nabla W(X_t - y; \theta_2)} f_\infty(y; \theta_2) \dd y \right)^2 \dd t,
\end{aligned}
\end{equation}
which by Cauchy--Schwarz inequality gives
\begin{equation}
\begin{aligned}
\abs{h(T, \theta_1) - h(T, \theta_2)}^2 &\le \frac{B}{T} \int_0^T \abs{\nabla V(X_t; \theta_1) - \nabla V(X_t;\theta_2)}^2 \dd t \\
&+ \frac{B}{T} \int_0^T \left( \int_{\R^d} \abs{\nabla W(X_t - y; \theta_1)}^2  (f_\infty(y; \theta_1) + f_\infty(y; \theta_2)) \dd y \right) \\
&\hspace{1.1cm} \times \left( \int_{\R^d} \abs{f_\infty(y; \theta_1) - f_\infty(y; \theta_2)} \dd y \right) \dd t \\
&+ \frac{B}{T} \int_0^T \int_{\R^d} \abs{\nabla W(X_t - y; \theta_1) - \nabla W(X_t - y; \theta_2)}^2 f_\infty(y; \theta_2) \dd y \dd t.
\end{aligned}
\end{equation}
Moreover, since $\nabla V$ and $\nabla W$ are Lipschitz in $\theta$ by \cref{as:MLE}\ref{it:Lipschitz} and polynomially bounded in $x$, using the uniform boundedness of the moments and the ergodic theorem, and due to the differentiability of the invariant measure with respect to $\theta$ \cite{BSV16}, we have that there exists a constant $L>0$, independent of $T$, such that
\begin{equation}
\abs{h(T, \theta_1) - h(T, \theta_2)}^2 \le L^2 \abs{\theta_1 - \theta_2}^2,
\end{equation}
which yields the desired result.
\end{proof}

We can finally prove \cref{thm:MLE_lin}.

\begin{proof}[Proof of~\cref{thm:MLE_lin}]
By definition, the estimator $\widetilde\theta_T$ satisfies for all $\theta \in \Theta$
\begin{equation}
\frac1T \widetilde\ell_T(\widetilde\theta_T) \ge \frac1T \widetilde\ell_T(\theta),
\end{equation}
and, in particular, taking $\theta = \theta_0$, we get by equation \eqref{eq:equivalent_likelihood_tilde}
\begin{equation}
\frac{\widetilde\ell_T(\widetilde\theta_T)}T \ge \frac{\widetilde\ell_T(\theta_0)}T = \frac\beta{2T} \int_0^T \left( \widetilde b_{\theta_0}(X_t) - b_{\theta_0}(t,X_t) \right)^\top \dd B_t - \frac\beta{4T} \int_0^T \abs{\widetilde b_{\theta_0}(X_t) - b_{\theta_0}(t,X_t)}^2 \dd t,
\end{equation}
which implies
\begin{equation}
\begin{aligned}
&\frac\beta{4T} \int_0^T \abs{\widetilde b_{\theta_0}(X_t) - b_{\theta_0}(t,X_t)}^2 \dd t + \left[ \frac\beta{2T} \int_0^T \left( \widetilde b_\theta(X_t) - b_{\theta_0}(t,X_t) \right)^\top \dd B_t \right]_{\theta = \widetilde\theta_T} \\
&\hspace{1.5cm} - \frac\beta{2T} \int_0^T \left( \widetilde b_{\theta_0}(X_t) - b_{\theta_0}(t,X_t) \right)^\top \dd B_t \ge \frac\beta{4T} \int_0^T \abs{\widetilde b_{\widetilde\theta_T}(X_t) - b_{\theta_0}(t,X_t)}^2 \dd t \ge 0.
\end{aligned}
\end{equation}
Applying \cref{cor:convergence_likelihood2,lem:convergence_likelihood1} we obtain
\begin{equation}
\lim_{T \to \infty} \frac1T \int_0^T \abs{\widetilde b_{\widetilde\theta_T}(X_t) - b_{\theta_0}(t,X_t)}^2 \dd t = 0 \qquad \text{a.s.}
\end{equation}
Let us now define the function $h = h(T,\theta)$ as in equation \eqref{eq:h_def}, which is Lipschitz in $\theta$, uniformly in $T$, by \cref{lem:f_Lipschitz}, and notice that
\begin{equation} \label{eq:convergence_f_thetat}
\lim_{T \to \infty} h(T, \widetilde \theta_T) = 0 \qquad \text{a.s.}
\end{equation}
Therefore, we deduce that $\widetilde\theta_T \to \Omega$ as $T \to \infty$ almost surely, where
\begin{equation}
\Omega = \left\{ \theta \in \Theta \colon \lim_{T \to \infty} h(T, \theta) = 0 \right\},
\end{equation}
which by definition \eqref{eq:h_def} and \cref{lem:convergence_likelihood2} coincides with
\begin{equation}
\Omega = \left\{ \theta \in \Theta \colon \int_{\R^d} \abs{\widetilde b_\theta(x) - \widetilde b_{\theta_0}(x)}^2 f_\infty(x;\theta_0) \dd x = 0 \right\}.
\end{equation}
Finally, by the identifiability condition in \cref{as:MLE}\ref{it:identifiability} we get that $\Omega = \{ \theta_0 \}$, and the desired result follows.
\end{proof}

\subsubsection{Numerical examples}

\begin{figure}
\begin{center}
\includegraphics{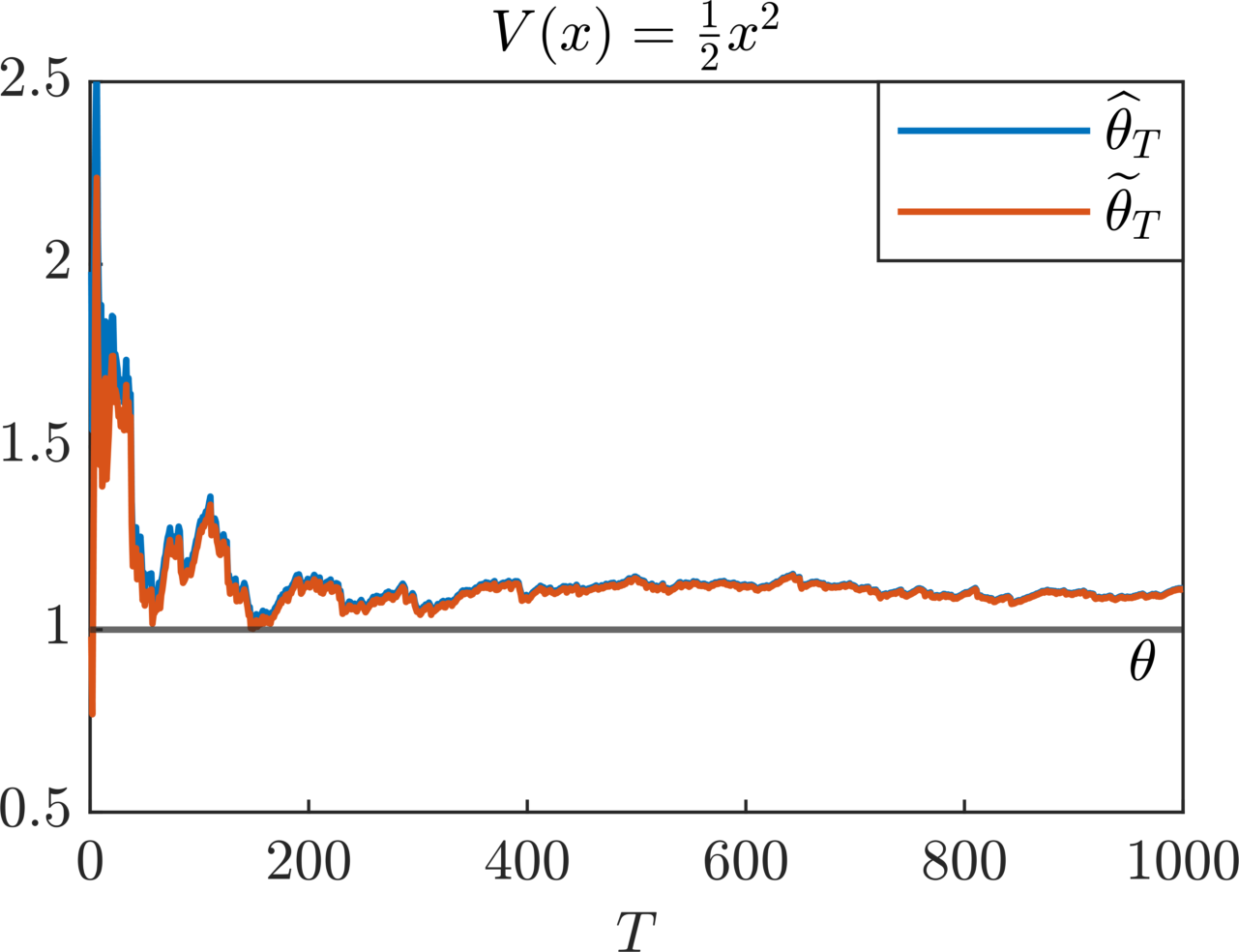} \hspace{1cm}
\includegraphics{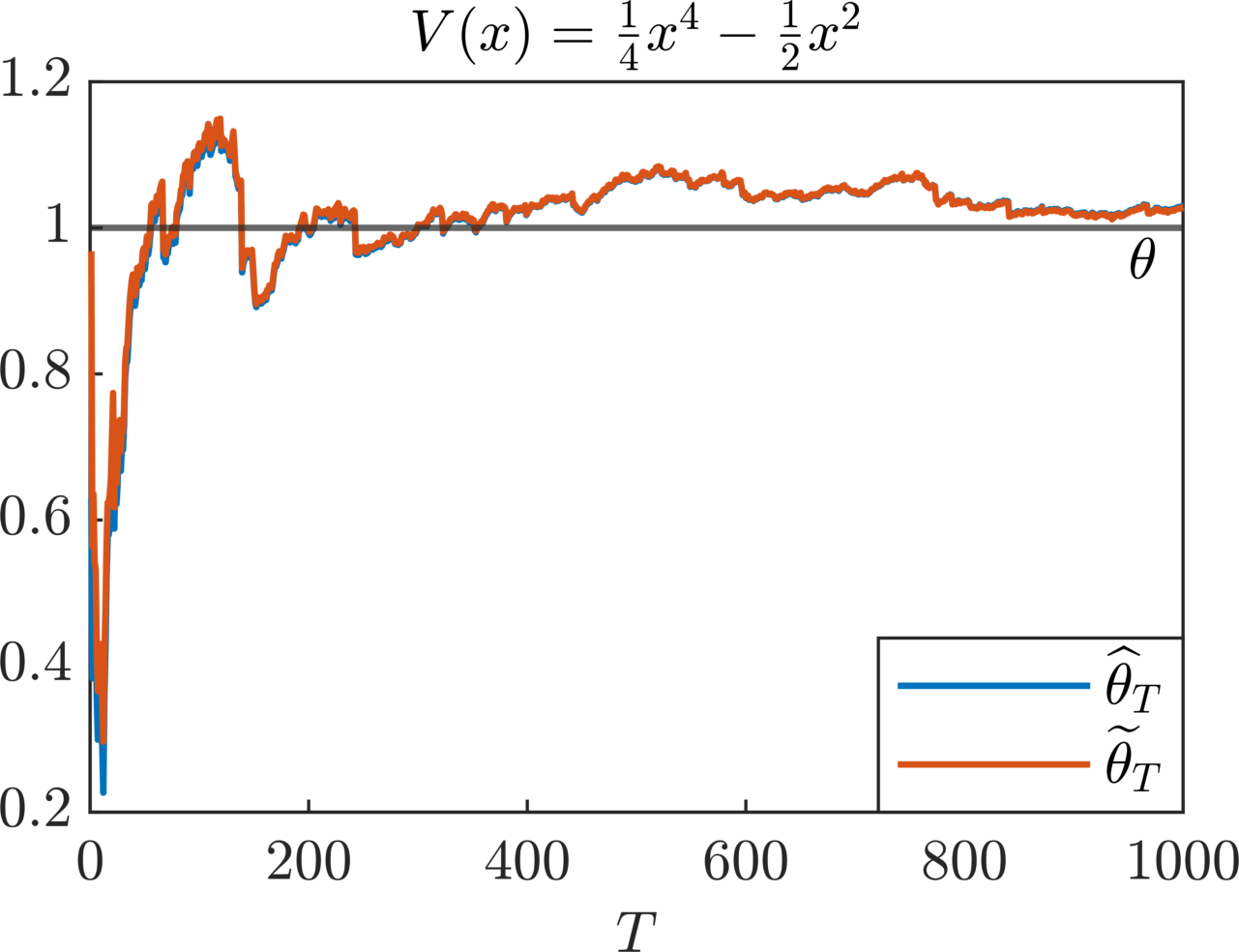}
\end{center}
\caption{Comparison between the MLE $\widehat \theta_T$ and the linearized MLE $\widetilde \theta_T$ for different values of the final time of observation $T \in [0, 10^3]$, considering the quadratic interaction potential and two different examples of confining potentials, when the unknown coefficient is in the confining potential.}
\label{fig:MLE}
\end{figure}

\begin{figure}
\begin{center}
\includegraphics{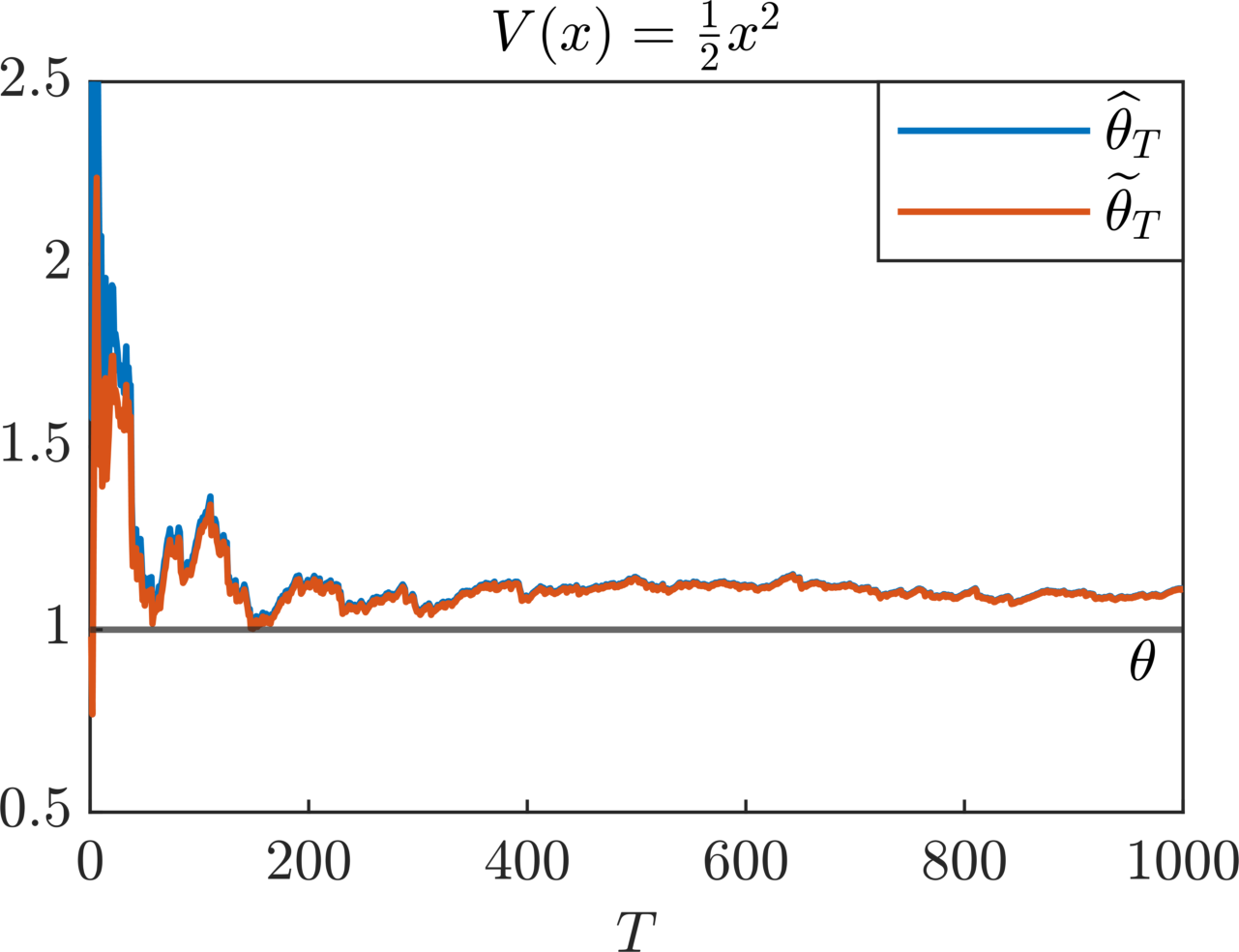} \hspace{1cm}
\includegraphics{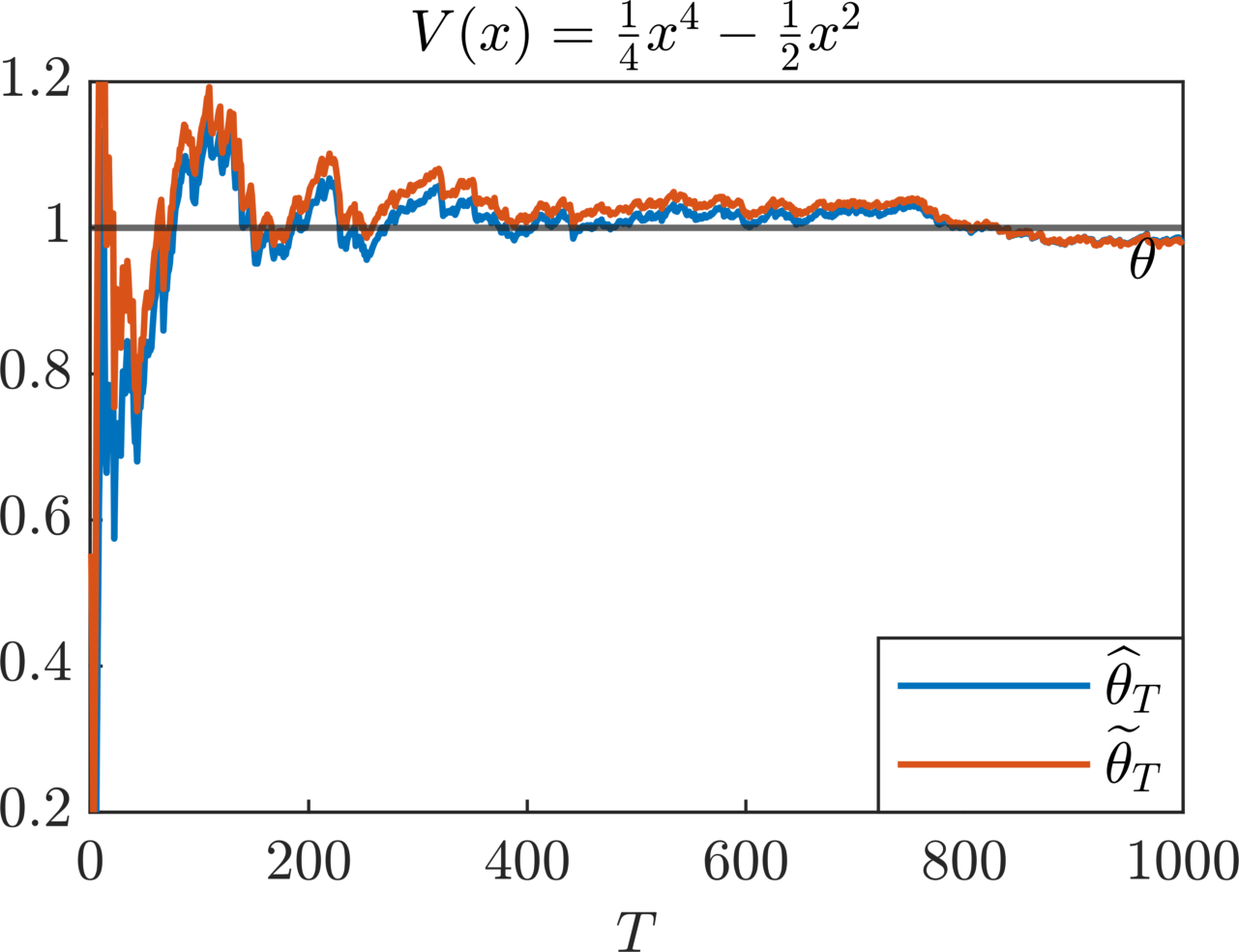}
\end{center}
\caption{Comparison between the MLE $\widehat \theta_T$ and the linearized MLE $\widetilde \theta_T$ for different values of the final time of observation $T \in [0, 10^3]$, considering the quadratic interaction potential and two different examples of confining potentials, when the unknown coefficient is in the interaction potential.}
\label{fig:MLE_interaction}
\end{figure}

We now implement the linearized MLE and test it on two numerical examples. In all the experiments presented in this section we consider deterministic initial conditions, $\mu_0 = \nu_0 = \delta_1$. For the discretization we use the Euler--Maruyama method with a time step $\Delta t = 10^{-3}$. We will also set $\beta = 1$. 

We first consider the one-dimensional mean field Ornstein--Uhlenbeck process, i.e., we set $V(x;\theta) = \theta x^2/2$ and $W(x;\theta) = x^2/2$ in equation \eqref{eq:McKeanSDE_MLE}, and therefore the McKean SDE reads
\begin{equation} \label{eq:OU}
\d X_t = - \theta X_t \dd t - (X_t - \E[X_t]) \dd t + \sqrt{2 \beta^{-1}} \dd B_t,
\end{equation}
where we set the exact value of the unknown parameter to be $\theta_0 = 1$. We take the expectation in~\eqref{eq:OU} to obtain
\begin{equation}
\E[X_t] = \E[X_0] - \theta \int_0^t \E[X_s] \dd s,
\end{equation}
which implies
\begin{equation}
\E[X_t] = \E[X_0] e^{- \theta t}.
\end{equation}
Therefore, the McKean SDE \eqref{eq:OU} can be rewritten as
\begin{equation}
\d X_t = - \theta X_t \dd t - (X_t - \E[X_0] e^{- \theta t}) \dd t + \sqrt{2 \beta^{-1}} \dd B_t,
\end{equation}
and, since its invariant measure is $\mu_\infty = \mathcal N(0, 1/(\beta(\theta+1)))$, the linearized process $Y_t$ satisfies
\begin{equation}
\d Y_t = - (\theta + 1) Y_t \dd t + \sqrt{2 \beta^{-1}} \dd B_t.
\end{equation}
Minimizing the log-likelihood $\mathfrak L_T$ and the linearized log-likelihood $\widetilde{\mathfrak L}_T$ in equations \eqref{eq:likelihood} and \eqref{eq:likelihood_linearized}, respectively, we deduce that the MLE $\widehat \theta_T$ solves the nonlinear equation
\begin{equation}
\begin{aligned}
0 &= - \int_0^T X_t \dd X_t - (\widehat \theta_T + 1) \int_0^T X_t^2 \dd t - \E[X_0] \int_0^T t e^{- \widehat \theta_T t} \dd X_t \\
&\quad + (\widehat \theta_T + 1) \E[X_0] \int_0^T t e^{- \widehat \theta_T t} X_t \dd t + \E[X_0] \int_0^T e^{- \widehat \theta_T t} X_t \dd t - \E[X_0]^2 \int_0^T t e^{-2 \widehat \theta_T t} \dd t,
\end{aligned}
\end{equation}
while the linearized MLE $\widetilde \theta_T$ is given by the closed form expression
\begin{equation}
\widetilde \theta_T = - \frac{\int_0^T X_t \dd X_t}{\int_0^T X_t^2 \dd t} - 1.
\end{equation}
In the first plot in \cref{fig:MLE}, we compare the two estimators for different values of the final time of observation $T \in [0, 10^3]$, and notice that, except for a short time interval, the difference between $\widehat \theta_T$ and $\widetilde \theta_T$ is negligible. In particular, they both provide an accurate approximation of the exact unknown.

In the second test case, we consider the bistable confining potential below the phase transition, i.e., when we have uniqueness of the invariant measure; see~\cite{Daw83, PaZ22} for details. Let $W(x;\theta) = x^2/2$ and $V(x;\theta) = \theta (x^4/4 - x^2/2)$ with true parameter $\theta_0 = 1$. Then, the invariant measure $\mu_\infty$ has density
\begin{equation}
f_\infty(x) = \frac{e^{- \theta \frac{x^4}4 - (1-\theta) \frac{x^2}2}}{\int_{\R} e^{- \theta \frac{z^4}4 - (1-\theta) \frac{z^2}2} \dd z}.
\end{equation}
The nonlinear and linear, respectively, processes $X_t$ and $Y_t$ solve the SDEs
\begin{equation}
\begin{aligned}
\d X_t &= - \theta (X_t^3 - X_t) \dd t - (X_t - \E[X_t]) \dd t + \sqrt{2 \beta^{-1}} \dd B_t, \\
\d Y_t &= - \theta Y_t^3 \dd t - (1 - \theta) Y_t \dd t + \sqrt{2 \beta^{-1}} \dd B_t.
\end{aligned}
\end{equation}
Assuming to know, in addition to the trajectory $(X_t)_{t \in [0,T]}$, also its expectation at all times $(\E[X_t])_{t \in [0,T]}$, we can write the MLE as
\begin{equation}
\widehat \theta_T = - \frac{\int_0^T (X_t^3 -X_t) \dd X_t + \int_0^T (X_t - \E[X_t])(X_t^3 -X_t) \dd t}{\int_0^T (X_t^3 - X_t)^2 \dd t}.
\end{equation}
On the other hand, the linearized MLE is given by the simpler expression
\begin{equation}
\widetilde \theta_T = - \frac{\int_0^T (X_t^3 -X_t) \dd X_t + \int_0^T X_t(X_t^3 -X_t) \dd t}{\int_0^T (X_t^3 - X_t)^2 \dd t},
\end{equation}
which does not need the expectation $\E[X_t]$. We note that, in this case, McKean's SDE must be approximated through the interacting particle system \eqref{eq:interacting_particle_system}. In particular, we use $N = 500$ particles, and a path of the mean field SDE is approximated by a path of one of the trajectories of the particle system. Due to exchangeability, it does not matter which particle we choose. We notice that assuming to know the expectation $\E[X_t]$ at all times implies that all the interacting particle system must be observed. On the contrary, for the linearized MLE, only the trajectory of a single particle is required. The behavior of the two estimators for different values of the final time of observation $T \in [0, 10^3]$ is shown in the second plot in \cref{fig:MLE}. As expected, the two estimators provide us with almost identical values for the estimated parameter, after a short initial transient time. 

We finally remark that similar results are obtained if the unknown parameter to be estimated appears in the interaction potential. Specifically, consider the mean field Ornstein--Uhlenbeck process together with its linearized process
\begin{equation}
\begin{aligned}
\d X_t &= - X_t \dd t - \theta (X_t - \E[X_t]) \dd t + \sqrt{2 \beta^{-1}} \dd B_t, \\
\d Y_t &= - (\theta + 1) Y_t \dd t + \sqrt{2 \beta^{-1}} \dd B_t.
\end{aligned}
\end{equation}
In this case we have $\E[X_t] = \E[X_0]e^{-t}$ and the MLE estimators are
\begin{equation}
\begin{aligned}
\widehat \theta_T &= - \frac{\int_0^T (X_t - \E[X_0]e^{-t}) \dd X_t + \int_0^T X_t (X_t - \E[X_0]e^{-t}) \dd t}{\int_0^T (X_t - \E[X_0]e^{-t})^2 \dd t}, \\
\widetilde \theta_T &= - \frac{\int_0^T X_t \dd X_t}{\int_0^T X_t^2 \dd t} - 1. 
\end{aligned}
\end{equation}
On the other hand, regarding the bistable confining potential example, the linear and nonlinear processes satisfy
\begin{equation}
\begin{aligned}
\d X_t &= - (X_t^3 - X_t) \dd t - \theta (X_t - \E[X_t]) \dd t + \sqrt{2 \beta^{-1}} \dd B_t, \\
\d Y_t &= - Y_t^3 \dd t - (\theta - 1) Y_t \dd t + \sqrt{2 \beta^{-1}} \dd B_t,
\end{aligned}
\end{equation}
and the corresponding MLE estimators read
\begin{equation}
\begin{aligned}
\widehat \theta_T &= - \frac{\int_0^T (X_t - \E[X_t]) \dd X_t + \int_0^T (X_t - \E[X_t])(X_t^3 -X_t) \dd t}{\int_0^T (X_t - \E[X_t])^2 \dd t}, \\
\widetilde \theta_T &= - \frac{\int_0^T X_t \dd X_t + \int_0^T X_t(X_t^3 -X_t) \dd t}{\int_0^T X_t^2 \dd t}.
\end{aligned}
\end{equation}
The numerical results are reported in \cref{fig:MLE_interaction}, where we observe that the linear and nonlinear MLE estimators agree and tend to coincide when the final time of observation increases.

\subsection{The diffusive-mean field limit} \label{sec:CLT}

We revisit the problem of studying the joint diffusive-mean field limit for weakly interacting diffusions in periodic confining and interaction potentials that was studied in \cite{DGP21}. As discussed in that paper, the two limits of sending $\epl \to 0$ and $N \to \infty$ commute only when the mean field dynamics on the torus does not exhibit phase transitions. Here, we study only the ``easy'' order of taking the two limits, first $N \to \infty$ and then sending $\epl$ to $0$, and only under the assumption that the mean field SDE has a unique steady state. A full analysis of the problem based on the linearized SDE, for both the overdamped and the underdamped dynamics, will be presented elsewhere.

We recall the McKean SDE~\eqref{eq:McKeanSDE}
\begin{equation} 
\d X_t = - \nabla V(X_t) \dd t - (\nabla W * f_t) (X_t) \dd t + \sqrt{2 \beta^{-1}} \dd B_t,
\end{equation}
where now the state space is $\R^d$, but the confining and interaction potentials $V$ and $W$ are assumed to be periodic with period the $d$-dimensional torus $\T^d$. First, notice that we have
\begin{equation}
(\nabla W * f_t)(x) = \int_{\R^d} \nabla W(x-y) f_t(y) \dd y = \sum_{k \in \Z^d} \int_{k + \T^d} \nabla W(x-y) f_t(y) \dd y,
\end{equation}
which, by the change of variables $z = y - k$ and the periodicity of $W$, implies
\begin{equation}
(\nabla W * f_t)(x) = \sum_{k \in \Z^d} \int_{\T^d} \nabla W(x-z-k) f_t(z+k) \dd z = \int_{\T^d} \nabla W(x-z) \sum_{k \in \Z^d} f_t(z+k) \dd z.
\end{equation}
Then, by defining the density $\phi_t \colon \T^d \to \R$ as
\begin{equation}
\phi_t(x) = \sum_{k \in \Z^d} f_t(x+k),
\end{equation}
we get 
\begin{equation}
(\nabla W * f_t)(x) = \int_{\T^d} \nabla W(x-z) \phi_t(z) \dd z = (\nabla W * \phi_t)(x).
\end{equation}
Moreover, considering the wrapping of the process $X_t$ on the torus, it is possible to prove that $\phi_t$ converges to the solution $\phi_\infty \colon \T^d \to \R$ of the stationary Fokker--Planck equation on $\T^d$
\begin{equation} \label{eq:invariant_phi}
\nabla \cdot \left( \left( \nabla V + \nabla W * \phi_\infty \right) \phi_\infty \right) + \beta^{-1} \Delta \phi_\infty = 0.
\end{equation}
This nonlinear and nonlocal elliptic PDE can be reformulated as the integral equation~\eqref{eq:integral_equation}; see~\cite{Tam84, Dre87}. In \cite{DGP21}, letting $\epl>0$, the diffusive rescaling $\widetilde X_t^\epl \defeq \epl X_{t/\epl^2}$ is considered and its limit as $\epl$ vanishes is studied. In particular, in \cite[Theorem 1.10]{DGP21}, under the assumption that the mean field dynamics on the torus is geometrically ergodic, the following functional central limit theorem (invariance principle) is proved. It holds
\begin{equation} \label{eq:FCLT}
\lim_{\epl\to0} \widetilde X_t^\epl = \sqrt{2\beta^{-1} \mathcal D} B_t, \qquad \text{in law},
\end{equation}
where $B_t$ is a standard $d$-dimensional Brownian motion and the matrix $\mathcal D$ is given by the standard homogenization formula~\cite[Chapter 13]{PaS08}
\begin{equation} \label{eq:Dhom_def}
\mathcal D = \E^{\phi_\infty}[(I + \nabla \Phi(X)) (I + \nabla \Phi(X))^\top],
\end{equation}
where the superscript denotes the fact that the expectation is computed with respect to the invariant measure with density $\phi_\infty$. The function $\Phi \colon \T^d \to \R^d$ in the equation above is the unique mean zero solution of the Poisson problem on $\T^d$ 
\begin{equation} \label{eq:cell_problem}
 \nabla \cdot \left(\phi_\infty \nabla \Phi  \right) = \nabla \phi_\infty,
\end{equation}
with periodic boundary conditions. The proof of this result is based on the standard methodology: application of It\^{o}'s formula to the solution of the (nolinear) Poisson equation, decomposition of the rescaled process into a martingale part and a reminder, use of PDE estimates on the Poisson equation together with the martingale central limit theorem. We remark that, for the proof of the invariance principle in our setting, it is sufficient to prove a central limit theorem of the form
\begin{equation} \label{eq:CLT}
\lim_{t \to \infty} \frac1{\sqrt t} X_t = \mathcal N(0, 2\beta^{-1} \mathcal D), \qquad \text{in law},
\end{equation}
as done in \cite[Lemma 2.3]{DGP21}. The invariance principle then follows from standard arguments; see, e.g.,~\cite{Oll94, KLO12}. We also remark that equation \eqref{eq:CLT} is also implied by equation \eqref{eq:FCLT}, by setting $t=1$ and $\epl = 1/\sqrt{t}$. Here we show that, following a similar but simpler argument, the same limit results can be obtained for the linearized process.

\subsubsection{Limit distribution of the linearized process}

In this section, we show that the linearized process behaves similarly in the limit $t\to\infty$, and, in particular, that the invariance principle and central limit theorem \eqref{eq:CLT} and \eqref{eq:FCLT}, respectively, are also valid for the linearized process, and with the same covariance matrix for the limiting Brownian motion. We make the appropriate assumptions on the drift and diffusion coefficients that ensure that the mean field process is geometrically ergodic and we consider the linearized process $Y_t$ around the unique invariant measure $\phi_\infty$
\begin{equation} \label{eq:linearized_Y}
\d Y_t = - \nabla V(Y_t) \dd t - (\nabla W * \phi_\infty) (Y_t) \dd t + \sqrt{2 \beta^{-1}} \dd B_t,
\end{equation}
and let $\widetilde Y_t^\epl \defeq \epl Y_{t/\epl^2}$. In the next two theorems we first prove the CLT for $Y_t$, whose argument is inspired by the proof of \cite[Lemma 2.3]{DGP21}, and then compute the limit of $\widetilde Y_t^\epl$ following the methodology of \cite{Oll94}. These results, which are analogous to equations \eqref{eq:CLT} and \eqref{eq:FCLT}, respectively, show that the limit distributions are the same as those for the nonlinear process.

\begin{theorem} \label{thm:CLT_linearized}
Let $Y_t$ be the solution of the SDE \eqref{eq:linearized_Y}. Then, it holds
\begin{equation}
\lim_{t \to \infty} \frac1{\sqrt t} Y_t = \mathcal N(0, 2\beta^{-1} \mathcal D), \qquad \text{in law},
\end{equation}
where $\mathcal D$ is defined in equation \eqref{eq:Dhom_def}.
\end{theorem}
\begin{proof}
Let $\Phi$ be given by equation \eqref{eq:cell_problem}. Applying Itô's lemma, we have
\begin{equation}
\begin{aligned}
\Phi(Y_t) &= \Phi(Y_0) + \sqrt{2\beta^{-1}} \int_0^t \nabla \Phi(Y_s) \dd B_s \\
&\quad - \int_0^t \nabla \Phi(Y_s) \left( \nabla V(Y_s) + (\nabla W * \phi_\infty)(Y_s) \right) \dd s  + \beta^{-1} \int_0^t \Delta \Phi(Y_s) \dd s,
\end{aligned}
\end{equation}
and by equation \eqref{eq:cell_problem} we obtain
\begin{equation}
- \int_0^t \left( \nabla V(Y_s) + (\nabla W * \phi_\infty)(Y_s) \right) \dd s = \Phi(Y_0) - \Phi(Y_t) + \sqrt{2\beta^{-1}} \int_0^t \nabla \Phi(Y_s) \dd B_s.
\end{equation}
Therefore, we get 
\begin{equation}
\begin{aligned}
Y_t &= Y_0 - \int_0^t \left( \nabla V(Y_s) + (\nabla W * \phi_\infty)(Y_s) \right) \dd s + \sqrt{2\beta^{-1}} \int_0^t \d B_s \\
&= Y_0 + \Phi(Y_0) - \Phi(Y_t) + \sqrt{2\beta^{-1}} \int_0^t (I + \nabla \Phi(Y_s)) \dd B_s,
\end{aligned}
\end{equation}
which implies
\begin{equation}
\frac1{\sqrt t} Y_t = \frac{Y_0 + \Phi(Y_0) - \Phi(Y_t)}{\sqrt t} + \sqrt{2\beta^{-1}} \frac1{\sqrt t} \int_0^t (I + \nabla \Phi(Y_s)) \dd B_s.
\end{equation}
Due to the boundedness of $\Phi$, we deduce that the first term in the right-hand side vanishes almost surely, and therefore in probability. Moreover, by the periodicity of $\Phi$ and the ergodic theorem we have
\begin{equation}
\lim_{t \to \infty} \frac1t \int_0^t (I + \nabla \Phi(Y_s)) (I + \nabla \Phi(Y_s))^\top \dd s = \E^{\phi_\infty}[(I + \nabla \Phi(Y)) (I + \nabla \Phi(Y))^\top] = \mathcal D.
\end{equation}
Finally, by the martingale central limit theorem we deduce that
\begin{equation}
\lim_{t \to \infty} \frac1{\sqrt t} \int_0^t (I + \nabla \Phi(Y_s)) \dd s = \mathcal N(0,\mathcal D), \qquad \text{in law},
\end{equation}
and the desired result follows by Slutsky's theorem. 
\end{proof} 

\begin{theorem} \label{thm:FCLT_linearized}
Let $Y_t$ be the solution of the SDE \eqref{eq:linearized_Y}, and let $\widetilde Y^\epl_t = \epl Y_{t/\epl^2}$. Then, it holds
\begin{equation}
\lim_{\epl \to 0} \widetilde Y^\epl_t = \sqrt{2\beta^{-1} \mathcal D} B_t, \qquad \text{in law},
\end{equation}
where $\mathcal D$ is defined in equation \eqref{eq:Dhom_def} and $B_t$ is a standard $d$-dimensional Brownian motion.
\end{theorem}
\begin{proof}
Reasoning as in the proof of \cref{thm:CLT_linearized}, we obtain
\begin{equation}
\widetilde Y^\epl_t = \epl Y_{t/\epl^2} = \epl \left( Y_0 + \Phi(Y_0) - \Phi(Y_{t/\epl^2}) \right) + \epl \sqrt{2\beta^{-1}} \int_0^{t/\epl^2} (I + \nabla \Phi(Y_s)) \dd B_s,
\end{equation}
where the first term in the right-hand side vanishes almost surely, and therefore in probability, as $\epl \to 0$ due to the boundedness of $\Phi$. Hence, the limiting properties of the rescaled process $\widetilde Y^\epl_t$ are determined studying the second part. Introducing the martingale
\begin{equation}
\mathcal M_\tau = \sqrt{2\beta^{-1}} \int_0^\tau (I + \nabla \Phi(Y_s)) \dd B_s,
\end{equation}
we can write
\begin{equation}
\widetilde Y^\epl_t = \epl \left( Y_0 + \Phi(Y_0) - \Phi(Y_{t/\epl^2}) \right) + \epl \mathcal M_{t/\epl^2}.
\end{equation}
By the ergodic theorem and the periodicity of $\Phi$, the quadratic variation $[\mathcal M]_{t/\epl^2}$ of the martingale term satisfies
\begin{equation} \label{eq:QV_M}
\lim_{\epl \to 0} [\epl \mathcal M]_{t/\epl^2} = 2 \beta^{-1} t \lim_{\epl \to 0} \frac{\epl^2}t \int_0^{t/\epl^2} (I + \nabla \Phi(Y_s)) (I + \nabla \Phi(Y_s))^\top \dd s = 2\beta^{-1} t \mathcal D.
\end{equation}
We now define the difference martingale 
\begin{equation}
\mathfrak M_{s,t}^\epl = \epl \mathcal M_{t/\epl^2} - \epl \mathcal M_{s/\epl^2},
\end{equation}
whose quadratic variation, reasoning as in equation \eqref{eq:QV_M}, satisfies
\begin{equation}
\lim_{\epl \to 0} [\mathfrak M^\epl]_{s,t} = 2 \beta^{-1} \lim_{\epl \to 0} \epl^2 \int_{s/\epl^2}^{t/\epl^2} (I + \nabla \Phi(Y_\tau)) (I + \nabla \Phi(Y_\tau))^\top \dd \tau = 2\beta^{-1}(t-s) \mathcal D.
\end{equation}
Finally, the desired result is obtained following the proof of \cite[Theorem 1.5]{Oll94}.
\end{proof}

\subsubsection{Numerical example}

\begin{figure}
\begin{center}
\includegraphics{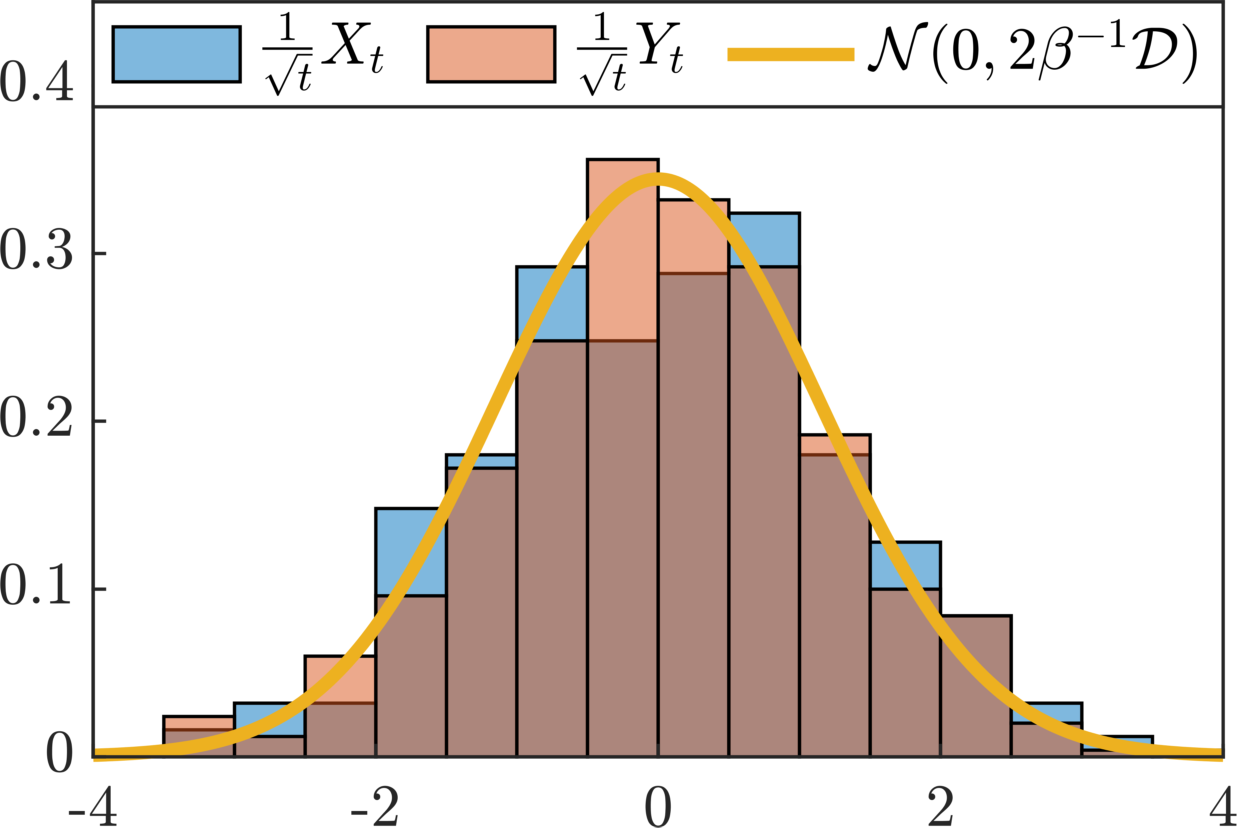}
\end{center}
\caption{Numerical verification of the CLTs in equation \eqref{eq:CLT} and \cref{thm:CLT_linearized}, considering sinusoidal confining and interaction potentials.}
\label{fig:CLT}
\end{figure}

We numerically verify in a one-dimensional test case the CLTs for the nonlinear and linear processes $X_t$ and $Y_t$ given in equations \eqref{eq:CLT} and \cref{thm:CLT_linearized}, respectively. We consider the setting of \cite[Lemma 1.26]{DGP21} below the phase transition, i.e., we set $W(x) = - \cos(2\pi x)$, $V(x) = - \xi \cos(2\pi x)$ with $\xi = 1/2$, and $\beta = 1$. In \cite[Lemma 1.26]{DGP21}, it is proved that the wrapping of the solution of the McKean SDE has a unique invariant measure with density
\begin{equation}
\phi_\infty(x) = \frac{e^{A \cos(2\pi x)}}{\int_{\T} e^{A \cos(2\pi z)} \dd z},
\end{equation}
where $A$ is the solution of the nonlinear equation
\begin{equation}
A = \beta \left(\xi + \frac{I_1(A)}{I_0(A)} \right),
\end{equation}
with $I_0$ and $I_1$ being modified Bessel functions of the first kind. The calculation of the limit diffusion coefficient, based upon the solution of the Poison equation~\eqref{eq:cell_problem}, is identical to what is done in the standard theory of periodic homogenization~\cite[Section 13.6.1]{PaS08}. Hence, we deduce that the limit coefficient $\mathcal D$ is given by
\begin{equation}
\mathcal D = \frac1{\left( \int_{\T} e^{A \cos(2\pi z)} \dd z \right)^2} = I_0(A)^{-2}.
\end{equation}
Moreover, equation \eqref{eq:linearized_Y} for the linearized process $Y_t$ can be rewritten as
\begin{equation} \label{eq:CLT_linear_Y}
\dd Y_t = - 2\pi \beta^{-1} A \sin(2\pi Y_t) \dd t + \sqrt{2\beta^{-1}} \dd B_t.
\end{equation}
In \cref{fig:CLT} we show the histograms of $X_t/\sqrt t$ and $Y_t/\sqrt t$ with $t = 10^3$ for $500$ different realizations of the Brownian motion. We remark that, differently from equation \eqref{eq:CLT_linear_Y} that can be simulated, the McKean SDE for $X_t$ must be approximated by the interacting particle system \eqref{eq:interacting_particle_system}, for which we use $N = 250$ particles. All the SDEs have initial distribution $\mu_0 = \nu_0 = \delta_0$, and they are discretized using the Euler--Maruyama method with a time step $\Delta t = 10^{-2}$. We observe that the histograms match the limit Gaussian distribution $\mathcal N(0, 2\beta^{-1} \mathcal D)$ predicted by the theory.

\section{Conclusion} \label{sec:conclusion}

In this work, we considered the long-time behaviour of the McKean SDE, obtained in the mean field limit of a system of weakly interacting diffusions. We showed that, under the appropriate assumptions on the confining and interaction potentials that ensure that the McKean SDE is geometrically ergodic, we can ``linearize'' the nonlinear process in a systematic way and derive an Itô diffusion process that behaves like the nonlinear process at long times. This linear (in the sense of McKean) SDE is obtained by replacing the law of the process with the invariant measure of the dynamics. We considered the problem both with $\R^d$ and $\T^d$ as state spaces, and studied the asymptotic behavior of the law of the linearized process. We showed its exponentially fast convergence in time to the law of the nonlinear process, even for different initial distributions, both in relative entropy and Wasserstein distance. The rate of convergence depends on the coefficients of the McKean SDE via the LSI constant of the corresponding Gibbs measure and on the LSI of the initial distribution. Our analysis strongly relies on LSIs for the law of the linearized process. Moreover, we also proved strong convergence of the paths in $L^2$ when the state space is $\R^d$. We remark that an additional contribution of our work is the study of the convergence to equilibrium for the McKean SDE on the torus with nonzero confining potential. In fact, under technical assumptions on the coefficients, we showed exponentially fast convergence to the invariant measure both in relative entropy and $L^1$ and $L^2$ for the densities.

Furthermore, we applied the linearization procedure to the study of the MLE for estimating unknown parameters in the McKean SDE. In particular, we considered the likelihood function of the linearized process, and proposed a ``linearized'' MLE that remains asymptotically unbiased in the limit of infinite data. We finally considered the combined diffusive-mean field limit for interacting particle systems in periodic interaction and confining potentials, and we verified that the limit distributions of the linear and nonlinear processes are the same. 

In the same spirit as we did in \cite{PaZ22}, where we implicitly used this approach, as future developments of this work, we plan to use the techniques developed in this paper to analyze nonparametric estimators for interacting particle systems. A crucial component of this program is the accurate approximation of the invariant measure, which is used to derive the linearized McKean SDE.  

The proposed methodology can be applied to other problems. We are particularly interested in the case where there are multiple invariant measures. We believe that our methodology can be applied to study the problem considered in~\cite{MoR24}. In addition, we would like to obtain similar results for the underdamped/kinetic McKean SDE and for the corresponding McKean--Vlasov PDE, as well as for coupled Vlasov--Poisson--Fokker--Planck systems. Another application of the proposed linearization procedure is the rigorous analysis of optimal control methodologies for the mean field dynamics that were developed in~\cite{BKP24}. All of these topics are currently under investigation.

\subsection*{Acknowledgements} 

We thank the anonymous reviewers whose comments and suggestions helped improve and clarify this manuscript. GAP is partially supported by an ERC-EPSRC Frontier Research Guarantee through grant no. EP/X038645, ERC through Advanced grant no. 247031, and a Leverhulme Trust Senior Research Fellowship, SRF$\backslash$R1$\backslash$241055. AZ is supported by ``Centro di Ricerca Matematica Ennio De Giorgi'' and the ``Emma e Giovanni Sansone'' Foundation, and is member of INdAM-GNCS.

\bibliographystyle{siamnodash}
\bibliography{biblio}

\begin{thebibliography}{10}

\bibitem{ABP24}
{\sc C.~Amorino, D.~Belomestny, V.~Pilipauskaitė, M.~Podolskij, and S.-Y.
  Zhou}, {\em Polynomial rates via deconvolution for nonparametric estimation
  in {M}c{K}ean-{V}lasov {SDE}s}.
\newblock Preprint arXiv:2401.04667, 2024.

\bibitem{BGL14}
{\sc D.~Bakry, I.~Gentil, and M.~Ledoux}, {\em Analysis and geometry of
  {M}arkov diffusion operators}, vol.~348 of Grundlehren der mathematischen
  Wissenschaften [Fundamental Principles of Mathematical Sciences], Springer,
  Cham, 2014.

\bibitem{Bav91}
{\sc F.~Bavaud}, {\em Equilibrium properties of the {V}lasov functional: the
  generalized {P}oisson-{B}oltzmann-{E}mden equation}, Rev. Modern Phys., 63
  (1991), pp.~129--148.

\bibitem{BeS23}
{\sc Z.~W. Bezemek and K.~Spiliopoulos}, {\em Rate of homogenization for
  fully-coupled {M}c{K}ean-{V}lasov {SDE}s}, Stoch. Dyn., 23 (2023), pp.~Paper
  No. 2350013, 65.

\bibitem{BKP24}
{\sc S.~Bicego, D.~Kalise, and G.~A. Pavliotis}, {\em Computation and control
  of unstable steady states for mean field multiagent systems}.
\newblock Preprint arXiv:2406.11725, 2024.

\bibitem{BiT08}
{\sc J.~Binney and S.~Tremaine}, {\em Galactic Dynamics}, Princeton University
  Press, Princeton, second~ed., 2008.

\bibitem{Bis11}
{\sc J.~P.~N. Bishwal}, {\em Estimation in interacting diffusions: continuous
  and discrete sampling}, Appl. Math. (Irvine), 2 (2011), pp.~1154--1158.

\bibitem{BRS16}
{\sc V.~I. Bogachev, M.~R\"ockner, and S.~V. Shaposhnikov}, {\em Distances
  between transition probabilities of diffusions and applications to nonlinear
  {F}okker-{P}lanck-{K}olmogorov equations}, J. Funct. Anal., 271 (2016),
  pp.~1262--1300.

\bibitem{BSV16}
{\sc V.~I. Bogachev, S.~V. Shaposhnikov, and A.~Y. Veretennikov}, {\em
  Differentiability of solutions of stationary {F}okker-{P}lanck-{K}olmogorov
  equations with respect to a parameter}, Discrete Contin. Dyn. Syst., 36
  (2016), pp.~3519--3543.

\bibitem{BoV05}
{\sc F.~Bolley and C.~Villani}, {\em Weighted
  {C}sisz\'{a}r-{K}ullback-{P}insker inequalities and applications to
  transportation inequalities}, Ann. Fac. Sci. Toulouse Math. (6), 14 (2005),
  pp.~331--352.

\bibitem{BoB83}
{\sc V.~Borkar and A.~Bagchi}, {\em Parameter estimation in continuous-time
  stochastic processes}, Stochastics, 8 (1982/83), pp.~193--212.

\bibitem{BPA24}
{\sc G.~Bruno, F.~Pasqualotto, and A.~Agazzi}, {\em Emergence of meta-stable
  clustering in mean-field transformer models}.
\newblock Preperint arXiv:2410.23228, 2024.

\bibitem{CDP20}
{\sc J.~A. Carrillo, M.~G. Delgadino, and G.~A. Pavliotis}, {\em A
  {$\lambda$}-convexity based proof for the propagation of chaos for weakly
  interacting stochastic particles}, J. Funct. Anal., 279 (2020), p.~108734.

\bibitem{CGP20}
{\sc J.~A. Carrillo, R.~S. Gvalani, G.~A. Pavliotis, and A.~Schlichting}, {\em
  Long-time behaviour and phase transitions for the {M}c{K}ean-{V}lasov
  equation on the torus}, Arch. Ration. Mech. Anal., 235 (2020), pp.~635--690.

\bibitem{CHS22}
{\sc J.~A. Carrillo, F.~Hoffmann, A.~M. Stuart, and U.~Vaes}, {\em
  Consensus-based sampling}, Stud. Appl. Math., 148 (2022), pp.~1069--1140.

\bibitem{CJZ24}
{\sc J.~A. Carrillo, S.~Jin, H.~Zhang, and Y.~Zhu}, {\em An interacting
  particle consensus method for constrained global optimization}.
\newblock Preprint arXiv:2405.00891, 2024.

\bibitem{CJL17}
{\sc B.~Chazelle, Q.~Jiu, Q.~Li, and C.~Wang}, {\em Well-posedness of the
  limiting equation of a noisy consensus model in opinion dynamics}, J.
  Differential Equations, 263 (2017), pp.~365--397.

\bibitem{CGL24}
{\sc F.~Comte, V.~Genon-Catalot, and C.~Larédo}, {\em Nonparametric moment
  method for scalar {M}c{K}ean-{V}lasov stochastic differential equations}.
\newblock Preprint hal-04460327v2, 2024.

\bibitem{Daw83}
{\sc D.~A. Dawson}, {\em Critical dynamics and fluctuations for a mean-field
  model of cooperative behavior}, J. Statist. Phys., 31 (1983), pp.~29--85.

\bibitem{Dea96}
{\sc D.~S. Dean}, {\em Langevin equation for the density of a system of
  interacting {L}angevin processes}, J. Phys. A, 29 (1996), pp.~L613--L617.

\bibitem{DeT24}
{\sc F.~Delarue and A.~Tse}, {\em Uniform in time weak propagation of chaos on
  the torus}.
\newblock Preprint arXiv:2104.14973, 2024.

\bibitem{DGP21}
{\sc M.~G. Delgadino, R.~S. Gvalani, and G.~A. Pavliotis}, {\em On the
  diffusive-mean field limit for weakly interacting diffusions exhibiting phase
  transitions}, Arch. Ration. Mech. Anal., 241 (2021), pp.~91--148.

\bibitem{DGP23}
{\sc M.~G. Delgadino, R.~S. Gvalani, G.~A. Pavliotis, and S.~A. Smith}, {\em
  Phase transitions, logarithmic {S}obolev inequalities, and uniform-in-time
  propagation of chaos for weakly interacting diffusions}, Comm. Math. Phys.,
  401 (2023), pp.~275--323.

\bibitem{DeH23}
{\sc L.~Della~Maestra and M.~Hoffmann}, {\em The {LAN} property for
  {M}c{K}ean-{V}lasov models in a mean-field regime}, Stochastic Process.
  Appl., 155 (2023), pp.~109--146.

\bibitem{Dre87}
{\sc K.~Dressler}, {\em Stationary solutions of the {V}lasov-{F}okker-{P}lanck
  equation}, Math. Methods Appl. Sci., 9 (1987), pp.~169--176.

\bibitem{DLP18}
{\sc M.~H. Duong, A.~Lamacz, M.~A. Peletier, A.~Schlichting, and U.~Sharma},
  {\em Quantification of coarse-graining error in {L}angevin and overdamped
  {L}angevin dynamics}, Nonlinearity, 31 (2018), pp.~4517--4566.

\bibitem{EGZ19}
{\sc A.~Eberle, A.~Guillin, and R.~Zimmer}, {\em Couplings and quantitative
  contraction rates for {L}angevin dynamics}, Ann. Probab., 47 (2019),
  pp.~1982--2010.

\bibitem{GPY17}
{\sc J.~Garnier, G.~Papanicolaou, and T.-W. Yang}, {\em Consensus convergence
  with stochastic effects}, Vietnam J. Math., 45 (2017), pp.~51--75.

\bibitem{GGK22}
{\sc B.~Gess, R.~S. Gvalani, and V.~Konarovskyi}, {\em Conservative spdes as
  fluctuating mean field limits of stochastic gradient descent}.
\newblock Preprint arXiv:2207.05705, 2022.

\bibitem{GGS21}
{\sc B.~D. Goddard, B.~Gooding, H.~Short, and G.~A. Pavliotis}, {\em Noisy
  bounded confidence models for opinion dynamics: the effect of boundary
  conditions on phase transitions}, IMA J. Appl. Math., 87 (2022), pp.~80--110.

\bibitem{Gol16}
{\sc F.~Golse}, {\em On the dynamics of large particle systems in the mean
  field limit}, in Macroscopic and large scale phenomena: coarse graining, mean
  field limits and ergodicity, vol.~3 of Lect. Notes Appl. Math. Mech.,
  Springer, [Cham], 2016, pp.~1--144.

\bibitem{GoP17}
{\sc S.~Gomes and G.~Pavliotis}, {\em Mean field limits for interacting
  diffusions in a two-scale potential}, J. Nonlin. Sci., 28 (2018),
  pp.~905--941.

\bibitem{GLM24}
{\sc A.~Guillin, P.~Le~Bris, and P.~Monmarché}, {\em Uniform in time
  propagation of chaos for the 2{D} vortex model and other singular stochastic
  systems}, J. Eur. Math. Soc.,  (2024).

\bibitem{Kas90}
{\sc R.~A. Kasonga}, {\em Maximum likelihood theory for large interacting
  systems}, SIAM J. Appl. Math., 50 (1990), pp.~865--875.

\bibitem{KaW98}
{\sc K.~Kawasaki}, {\em Microscopic analyses of the dynamical density
  functional equation of dense fluids}, J. Statist. Phys., 93 (1998),
  pp.~527--546.

\bibitem{KeS99}
{\sc M.~Kessler and M.~S\o~rensen}, {\em Estimating equations based on
  eigenfunctions for a discretely observed diffusion process}, Bernoulli, 5
  (1999), pp.~299--314.

\bibitem{KLO12}
{\sc T.~Komorowski, C.~Landim, and S.~Olla}, {\em Fluctuations in {M}arkov
  processes}, vol.~345 of Grundlehren der mathematischen Wissenschaften
  [Fundamental Principles of Mathematical Sciences], Springer, Heidelberg,
  2012.
\newblock Time symmetry and martingale approximation.

\bibitem{KLR19}
{\sc V.~Konarovskyi, T.~Lehmann, and M.-K. von Renesse}, {\em Dean-{K}awasaki
  dynamics: ill-posedness vs. triviality}, Electron. Commun. Probab., 24
  (2019), pp.~Paper No. 8, 9.

\bibitem{LaL23}
{\sc D.~Lacker and L.~Le~Flem}, {\em Sharp uniform-in-time propagation of
  chaos}, Probab. Theory Related Fields, 187 (2023), pp.~443--480.

\bibitem{LiQ22}
{\sc M.~Liu and H.~Qiao}, {\em Parameter estimation of path-dependent
  {M}c{K}ean-{V}lasov stochastic differential equations}, Acta Math. Sci. Ser.
  B (Engl. Ed.), 42 (2022), pp.~876--886.

\bibitem{Mal01}
{\sc F.~Malrieu}, {\em Logarithmic {S}obolev inequalities for some nonlinear
  {PDE}'s}, Stochastic Process. Appl., 95 (2001), pp.~109--132.

\bibitem{MMN18}
{\sc S.~Mei, A.~Montanari, and P.-M. Nguyen}, {\em A mean field view of the
  landscape of two-layer neural networks}, Proc. Natl. Acad. Sci. USA, 115
  (2018), pp.~E7665--E7671.

\bibitem{MRW24}
{\sc P.~Monmarché, Z.~Ren, and S.~Wang}, {\em Time-uniform log-{S}obolev
  inequalities and applications to propagation of chaos}.
\newblock Preprint arXiv:2401.07966, 2024.

\bibitem{MoR24}
{\sc P.~Monmarché and J.~Reygner}, {\em Local convergence rates for
  {W}asserstein gradient flows and {M}c{K}ean-{V}lasov equations with multiple
  stationary solutions}.
\newblock Preprint arXiv:2404.15725, 2024.

\bibitem{NPR24}
{\sc R.~Nickl, G.~A. Pavliotis, and K.~Ray}, {\em Bayesian nonparametric
  inference in {M}c{K}ean-{V}lasov models}.
\newblock Preprint arXiv:2404.16742, 2024.

\bibitem{Oel84}
{\sc K.~Oelschl\"ager}, {\em A martingale approach to the law of large numbers
  for weakly interacting stochastic processes}, Ann. Probab., 12 (1984),
  pp.~458--479.

\bibitem{Oll94}
{\sc S.~Olla}, {\em Homogenization of Diffusion Processes in Random Fields},
  Ecole Polytechnique, 1994.

\bibitem{OtV00}
{\sc F.~Otto and C.~Villani}, {\em Generalization of an inequality by
  {T}alagrand and links with the logarithmic {S}obolev inequality}, J. Funct.
  Anal., 173 (2000), pp.~361--400.

\bibitem{PaS08}
{\sc G.~Pavliotis and A.~Stuart}, {\em Multiscale methods}, vol.~53 of Texts in
  Applied Mathematics, Springer, New York, 2008.
\newblock Averaging and homogenization.

\bibitem{Pav14}
{\sc G.~A. Pavliotis}, {\em Stochastic processes and applications}, vol.~60 of
  Texts in Applied Mathematics, Springer, New York, 2014.
\newblock Diffusion processes, the Fokker-Planck and Langevin equations.

\bibitem{PaZ22}
{\sc G.~A. Pavliotis and A.~Zanoni}, {\em Eigenfunction martingale estimators
  for interacting particle systems and their mean field limit}, SIAM J. Appl.
  Dyn. Syst., 21 (2022), pp.~2338--2370.

\bibitem{PaZ24}
{\sc G.~A. Pavliotis and A.~Zanoni}, {\em A method of moments estimator for
  interacting particle systems and their mean field limit}, SIAM/ASA J.
  Uncertain. Quantif., 12 (2024), pp.~262--288.

\bibitem{RRW22}
{\sc P.~Ren, M.~R\"ockner, and F.-Y. Wang}, {\em Linearization of nonlinear
  {F}okker-{P}lanck equations and applications}, J. Differential Equations, 322
  (2022), pp.~1--37.

\bibitem{RoV22}
{\sc G.~M. Rotskoff and E.~Vanden-Eijnden}, {\em Trainability and accuracy of
  artificial neural networks: an interacting particle system approach}, Comm.
  Pure Appl. Math., 75 (2022), pp.~1889--1935.

\bibitem{Rue99}
{\sc D.~Ruelle}, {\em Statistical mechanics}, World Scientific Publishing Co.,
  Inc., River Edge, NJ; Imperial College Press, London, 1999.
\newblock Rigorous results, Reprint of the 1989 edition.

\bibitem{SKP23}
{\sc L.~Sharrock, N.~Kantas, P.~Parpas, and G.~A. Pavliotis}, {\em Online
  parameter estimation for the {M}c{K}ean-{V}lasov stochastic differential
  equation}, Stochastic Process. Appl., 162 (2023), pp.~481--546.

\bibitem{SiS20}
{\sc J.~Sirignano and K.~Spiliopoulos}, {\em Mean field analysis of neural
  networks: a law of large numbers}, SIAM J. Appl. Math., 80 (2020),
  pp.~725--752.

\bibitem{Suz05}
{\sc T.~Suzuki}, {\em Free energy and self-interacting particles}, vol.~62 of
  Progress in Nonlinear Differential Equations and their Applications,
  Birkh\"{a}user Boston, Inc., Boston, MA, 2005.

\bibitem{Szn91}
{\sc A.-S. Sznitman}, {\em Topics in propagation of chaos}, in \'{E}cole
  d'\'{E}t\'{e} de {P}robabilit\'{e}s de {S}aint-{F}lour {XIX}---1989,
  vol.~1464 of Lecture Notes in Math., Springer, Berlin, 1991, pp.~165--251.

\bibitem{Tam84}
{\sc Y.~Tamura}, {\em On asymptotic behaviors of the solution of a nonlinear
  diffusion equation}, J. Fac. Sci. Univ. Tokyo Sect. IA Math., 31 (1984),
  pp.~195--221.

\bibitem{Tsy09}
{\sc A.~B. Tsybakov}, {\em Introduction to nonparametric estimation}, Springer
  Series in Statistics, Springer, New York, 2009.
\newblock Revised and extended from the 2004 French original, Translated by
  Vladimir Zaiats.

\bibitem{Tug14}
{\sc J.~Tugaut}, {\em Self-stabilizing processes in multi-wells landscape in
  {$\Bbb{R}^d$}-invariant probabilities}, J. Theoret. Probab., 27 (2014),
  pp.~57--79.

\bibitem{Van01}
{\sc J.~H. van Zanten}, {\em A note on consistent estimation of multivariate
  parameters in ergodic diffusion models}, Scand. J. Statist., 28 (2001),
  pp.~617--623.

\bibitem{Zha23}
{\sc S.-Q. Zhang}, {\em Existence and non-uniqueness of stationary
  distributions for distribution dependent {SDE}s}, Electron. J. Probab., 28
  (2023), pp.~Paper No. 93, 34.

\end{thebibliography}

\end{document}